\documentclass{amsart}
\usepackage{amsfonts}
\usepackage{amsthm}
\usepackage{amsmath}
\usepackage{amscd}
\usepackage{amssymb}
\usepackage{amsthm}
\usepackage{color}
\usepackage{enumerate}
\usepackage{mathrsfs}
\usepackage[latin2]{inputenc}
\usepackage{t1enc}
\usepackage{indentfirst}
\usepackage{graphicx}
\usepackage{graphics}
\usepackage{pict2e}
\usepackage{epic}
\numberwithin{equation}{section}
\usepackage[margin=3.0cm]{geometry}
\usepackage{hyperref}
\usepackage{epstopdf}
\usepackage{tikz-cd} 
\usepackage{cleveref}

\theoremstyle{plain}
\newtheorem{Th}{Theorem}[section]
\newtheorem{Lemma}[Th]{Lemma}
\newtheorem{Cor}[Th]{Corollary}
\newtheorem{Prop}[Th]{Proposition}

\theoremstyle{definition}

\newtheorem{Rem}[Th]{Remark}
\newtheorem{?}[Th]{Problem}
\newtheorem{Ex}[Th]{Example}

\newcommand{\Aut}{\textrm{Aut}}
\newcommand{\Gal}{\textrm{Gal}}
\newcommand{\Q}{\mathbb{Q}}
\newcommand{\Z}{\mathbb{Z}}

\newcommand{\R}{\mathbb{R}}

\newcommand{\C}{\mathbb{C}}
\renewcommand{\P}{\mathbb{P}}

\newcommand{\ov}{\overline}

\newcommand{\mf}{\mathfrak}
\renewcommand{\phi}{\varphi}
\newcommand{\lan}{\langle}
\newcommand{\ran}{\rangle}

\newcommand{\mc}{\mathcal}
\newcommand{\wt}{\widetilde}

\newcommand{\spec}{\text{Spec}}

\newcommand{\im}{\operatorname{im}}

\begin{document}

\title{Distribution of the successive minima of the Petersson norm on cusp forms}
\author{Souparna Purohit}
\date{\today}

\maketitle 

\begin{abstract}
Let $\Gamma \subseteq \text{PSL}_2(\Z)$ be a finite index subgroup. Let $\mathscr{X}(\Gamma)$ be a regular proper model of the modular curve associated with $\Gamma$, and let $\ov{\mathscr{L}}^{\otimes k}$ be the logarithmically singular metrized line bundle on $\mathscr{X}(\Gamma)$ associated to modular forms of level $\Gamma$ and weight $12k$, endowed with the Petersson metric. For each $k \geq 1$, the sub-lattice $\mathscr{S}_k \subseteq H^0(\mathscr{X}(\Gamma), \mathscr{L}^{\otimes k})$ of integral cusp forms of level $\Gamma$ and weight $12k$ is a euclidean lattice with respect to the Petersson norm. In this paper, we describe the distribution of the successive minima of the $\mathscr{S}_k$ as $k \to \infty$, generalizing the work of Chinburg, Guignard, and Soul\'{e} in \cite{TGS}, which addressed the case $\Gamma = \text{PSL}_2(\Z)$. 
\end{abstract}

\setcounter{tocdepth}{1}
\tableofcontents

\section{Introduction}

Let $\mathscr{X}$ be an arithmetic variety (integral, flat, projective scheme) over $\spec(\Z)$ with smooth generic fiber. Let $\ov{\mathscr{L}}$ be a possibly singular hermitian line bundle on $\mathscr{X}$, and for every integer $k \geq 1$, let $\ov{\mathscr{L}}^{\otimes k}$ denote $\mathscr{L}^{\otimes k}$ endowed with the product metric. Let $\| \cdot \|_{k, \infty}$ denote the $L^2$ norm on the global sections $H^0(\mathscr{X}_\C, \mathscr{L}_\C^{\otimes k}) := H^0(\mathscr{X}, \mathscr{L}^{\otimes k}) \otimes_\Z \C$ with respect to a fixed smooth volume form on $\mathscr{X}(\C)$ invariant under complex conjugation. The finite free $\Z$-modules $H^0(\mathscr{X}, \mathscr{L}^{\otimes k})$ endowed with $\| \cdot \|_{k, \infty}$ are euclidean lattices.  

A basic set of invariants of euclidean lattices are their successive minima. Recall that for $i = 1, \dots, d_k$, where $d_k := \text{rk}_\Z \ H^0(\mathscr{X}, \mathscr{L}^{\otimes k})$, the \textit{$i$th successive minima}, $\mu_{k,i}$, is the smallest real number $\mu$ such that there exist vectors $s_1, \dots, s_i \in H^0(\mathscr{X}, \mathscr{L}^{\otimes k})$ that are linearly independent in $H^0(\mathscr{X}, \mathscr{L}^{\otimes k}) \otimes_\Z \Q$ and such that $\| s_j \|_{k, \infty} \leq \mu$ for all $j$. The study of successive minima has a long history, beginning with Minkowski's work on the geometry of numbers, to more recent works in Arakelov theory. From the perspective of Arakelov theory, it turns out to be more natural to work with $\lambda_{k,i} := -\log(\mu_{k,i})$, called the \textit{successive maxima}. 

We are interested in the asymptotic distribution of the successive maxima of $H^0(\mathscr{X}, \mathscr{L}^{\otimes k})$ as $k \to \infty$. When the metric on $\ov{\mathscr{L}}$ is smooth, or even continuous, the work of Chen in \cite{Chen-1} shows the existence of a limiting distribution associated to the successive maxima of the sequence of lattices $H^0(\mathscr{X}, \mathscr{L}^{\otimes k})$. More precisely, if 
\[
\nu_k := \frac{1}{d_k} \sum_{i=1}^{d_k} \delta_{\frac{1}{k} \lambda_{k,i}}, 
\]
where $\delta_x$ is the Dirac delta function at a real number $x$, then Chen's results (in particular, theorem 4.1.8 of \cite{Chen-1}) show that the sequence of probability measures $(\nu_k)_k$ on $\R$ converges weakly to a \textit{compactly supported} probability measure on $\R$. 

In \cite{TGS}, Chinburg, Guignard, and Soul\'{e} prove an analogous result when the metric on $\ov{\mathscr{L}}$ is \textit{logarithmically singular} (see \cite{Kuhn} definition 3.1, and \cite{BKK} chapter 7). In particular, they consider $\mathscr{X} = \mathscr{X}(1)$ an integral model of the modular curve $X(1)_\C$ for $\text{PSL}_2(\Z)$, and $\ov{\mathscr{L}}$ the line bundle of integral modular forms for the group $\text{PSL}_2(\Z)$ of weight $12$ with the Petersson metric (see \S 3 of \cite{TGS}). As shown by K\"{u}hn in proposition 4.9 of \cite{Kuhn}, the Petersson metric on the line bundle of modular forms is logarithmically singular along the cusps and elliptic points. Due to these singularities, the $L^2$ norms, i.e. the Petersson norms, are no longer defined on the spaces $H^0(\mathscr{X}(1)_\C, \mathscr{L}_\C^{\otimes k})$ of all modular forms of weight $12k$. But they are defined on the subspaces $H^0(\mathscr{X}(1)_\C, \mathscr{L}_\C^{\otimes k}(-\infty))$ of cusp forms (here $\infty \in \mathscr{X}(1)_\C$ denotes the unique cusp). In this setting, theorem 3.2.2 of \cite{TGS} describes the asymptotic behavior of the successive maxima of the lattices of integral cusp forms for the group $\text{PSL}_2(\Z)$ of weight $12k$ with the Petersson norm as $k \to \infty$, in close analogy to Chen's result mentioned above. 

In this paper, we generalize theorem 3.2.2 of \cite{TGS} to cusp forms of any finite index subgroup $\Gamma \subseteq \text{PSL}_2(\Z)$. We state our results precisely in the following sub-section. 

\subsection{Statement of main results}
\label{statement-of-main-results-intro}
We follow the setup in \S 4.11 of \cite{Kuhn}. Let $\Gamma(1) := \text{PSL}_2(\Z)$, and identify the complex modular curve $X(1)_\C$ associated to $\Gamma(1)$ with $\P_\C^1$ via the modular $j$-function (see example \ref{j-function-example}). Let $\Gamma \subseteq \Gamma(1)$ be a finite index subgroup. The modular curve $X(\Gamma)_\C$ and the natural map $\pi_{\Gamma, \C} : X(\Gamma)_\C \to \P_\C^1$ are defined over a number field $E$ (see \S \ref{integral-models-ss}). Let $X(\Gamma)$ be a (smooth projective geometrically connected) model of $X(\Gamma)_\C$ over $E$, let $\pi_{\Gamma, E} : X(\Gamma) \to \P_E^1$ be the model of $\pi_{\Gamma, \C}$ over $E$, and let $\pi_\Gamma : X(\Gamma) \to \P_\Q^1$ denote the composition of $\pi_{\Gamma, E}$ with the natural map from $\P_E^1$ to $\P_\Q^1$. Let $\infty \in \P_\Q^1$ correspond to the unique pole of the $j$-function, and let $D \in \text{Div}(X(\Gamma))$ denote the sum of the points in $\pi_{\Gamma}^{-1}(\infty)$. We call $D$ the \textit{divisor of cusps of $X(\Gamma)$}.  

Let $\mathscr{X}(\Gamma)$ be \textit{an arithmetic surface associated to $\Gamma$} (see \S \ref{integral-models-ss} and \S 4.11 of \cite{Kuhn}). Then $\mathscr{X}(\Gamma)$ is a regular projective arithmetic surface with generic fiber $X(\Gamma)$, and comes with a morphism $\pi_{\Gamma, \Z} : \mathscr{X}(\Gamma) \to \P_\Z^1$ extending $\pi_\Gamma$. Let $\ov{\mathscr{L}}$ denote the metrized line bundle on $\mathscr{X}(\Gamma)$ associated to modular forms of level $\Gamma$ and weight $12$ endowed with the Petersson metric (see \S \ref{integral-models-ss} and \S \ref{rational-Pet-met-ss}, and \S 4.12 of \cite{Kuhn}). As mentioned above, this metric has logarithmic singularities at elliptic points and cusps (proposition 4.9 in \cite{Kuhn}). 

\begin{Th}
\label{main-theorem-statement-intro}
Let $\Gamma \subseteq \Gamma(1)$ be a finite index subgroup. Let $\mathscr{X}(\Gamma)$ be an arithmetic surface associated to $\Gamma$, $X(\Gamma)$ its generic fiber, and let $D$ denote the divisor of cusps of $X(\Gamma)$. Let $\ov{\mathscr{L}}$ be the line bundle on $\mathscr{X}(\Gamma)$ associated to modular forms of level $\Gamma$ and weight $12$ endowed with the Petersson metric. For every $k \geq 1$, let
\[
\mathscr{S}_k := H^0(\mathscr{X}(\Gamma), \mathscr{L}^{\otimes k}) \cap H^0(X(\Gamma), \mathscr{L}_\Q^{\otimes k}(-D)) 
\]
denote the euclidean lattice of integral cusp forms of level $\Gamma$ and weight $12k$ with respect to the Petersson inner product. Let $\mu_{k,i}$ denote the successive minima of $\mathscr{S}_k$, and let $\lambda_{k,i} := -\log(\mu_{k,i})$ denote the successive maxima. Let $d_k := \emph{rk}_\Z \mathscr{S}_k$, and let 
\[
\nu_k := \frac{1}{d_k} \sum_{i=1}^{d_k} \delta_{\frac{1}{k} \lambda_{k,i}}
\]
denote the probability measure on $\R$ associated to the normalized successive maxima of $\mathscr{S}_k$. Then the $\nu_k$ converge weakly to a Borel probability measure $\nu$ on $\R$. $\nu$ has support bounded above and unbounded below. 
\end{Th}

A natural question is: given finite index subgroups $\Gamma' \subseteq \Gamma$ of $\Gamma(1)$, how do the successive maxima of their integral cusp forms, and the associated limit measures compare? This is addressed by the following theorem. 
\begin{Th}
\label{measure-comparison-thm-intro}
Let $\Gamma' \subseteq \Gamma$ be finite index subgroups of $\Gamma(1)$. Let $\mathscr{X}(\Gamma')$ (resp. $\mathscr{X}(\Gamma)$) be an arithmetic surface associated to $\Gamma'$ (resp. $\Gamma$) with generic fiber $X(\Gamma')$ (resp. $X(\Gamma)$). Suppose there is a $\Q$-morphism $\pi_{\Gamma', \Gamma} : X(\Gamma') \to X(\Gamma)$ that is a model over $\Q$ of the natural map $X(\Gamma')_\C \to X(\Gamma)_\C$ of complex modular curves. 

Let $\nu_k'$ (resp. $\nu_k$) denote the probability measure on $\R$ associated to the normalized successive maxima of the euclidean lattice of integral cusp forms of level $\Gamma'$ (resp. $\Gamma$) and weight $12k$ with respect to $\mathscr{X}(\Gamma')$ (resp. $\mathscr{X}(\Gamma)$) as in theorem \ref{main-theorem-statement-intro}. Suppose that $\nu_k' \to \nu'$ (resp. $\nu_k \to \nu$) weakly for a Borel probability measure $\nu'$ (resp. $\nu$) on $\R$ (by theorem \ref{main-theorem-statement-intro}). Then there is a Borel probability measure $\omega$ on $\R$ such that 
\[
\nu' = \frac{1}{\emph{deg}(\pi_{\Gamma', \Gamma})} \cdot \nu + \left(1 - \frac{1}{\emph{deg}(\pi_{\Gamma', \Gamma})} \right) \cdot \omega. 
\] 
\end{Th}

\subsection{Remarks}
\begin{Rem}
As mentioned before, theorem \ref{main-theorem-statement-intro} generalizes theorem 3.2.2 (i), (ii) of \cite{TGS}, which addresses the case $\Gamma = \Gamma(1)$ for the associated arithmetic surface $\mathscr{X}(1) = \P_\Z^1$ (identification coming from the $j$-function). We remark that our overall approach is similar to that of \cite{TGS}, with a couple of key differences. First, the approach in \cite{TGS} uses various properties of $q$-expansions of modular forms for $\Gamma(1)$, including integrality of the coefficients of $q$-expansions of modular forms over the integral modular curve $\mathscr{X}(1)$ to obtain (lower) bounds on the Petersson norms of integral cusp forms (see for instance, proposition 3.3.1, lemma 3.3.1, theorem 3.4.1, and lemma 3.4.2 in \cite{TGS}). For a general (in particular, non-congruence) finite index subgroup $\Gamma \subseteq \Gamma(1)$, we do not have access to such integrality properties for the $q$-expansion coefficients of integral modular forms. 

Instead, following \cite{Chen-1}, we use results from the intersection theory of logarithmically singular line bundles on arithmetic surfaces as developed by K\"{u}hn in \cite{Kuhn}, along with height formulas developed by Bost, Gillet, and Soul\'{e} in \cite{Bost-Gillet-Soule} to obtain corresponding (lower) bounds for the sup norms of integral cusp forms (see proposition \ref{height-upper-bound-prop}, which is the counterpart to lemma 3.4.2 in \cite{TGS}). We then appeal to a ``Gromov's lemma'' type result in our log-singular setting (see proposition \ref{Gromov-lemma-prop}, page 1 of \cite{FJK}, and theorem 1.7 in \cite{AMM-Bergman-Kernels}), comparing sup norms to $L^2$ norms for sections in increasing powers of the line bundle of modular forms, to get analogous (lower) bounds on the $L^2$ (i.e. Petersson) norms of integral cusp forms for $\Gamma$. The estimates obtained from the intersection theory approach, however, are not as sharp as those obtained from the $q$-expansions, and this ultimately prevents us from showing directly that the support of the limit measure $\nu$ in theorem \ref{main-theorem-statement-intro} is unbounded below. Instead, we deduce unboundedness (in corollary \ref{unboundedness-general-corollary}) from theorem \ref{measure-comparison-thm-intro} along with the unboundedness in the case $\Gamma = \Gamma(1)$ proved in part (ii) of theorem 3.2.2 of \cite{TGS}. 

Another key difference with \cite{TGS} is that, unlike the $\Gamma = \Gamma(1)$ case considered there, it is difficult to write down explicit bases for the spaces of cusp forms for general $\Gamma$, which are then used for getting bounds on the successive maxima. In proposition \ref{lower-bounds-successive-max}, we construct less explicit bases that nevertheless have the same general shape as those in \cite{TGS} theorem 3.4.1 and lemma 3.4.4, and that turns out to be enough to yield the desired bounds on the successive maxima in our case. 
\end{Rem}

\begin{Rem}
An important feature of theorem \ref{main-theorem-statement-intro} and theorem 3.2.2 of \cite{TGS} is the fact that the limit measures have support unbounded below, in stark contrast to Chen's theorem for continuous metrics mentioned before, where the limit measure has compact support. This is explained, very roughly, as follows. When $\mathscr{L}$ has continuous metric, then at least if $\mathscr{L}_\Q$ is ample, 
\[
\bigoplus_{k \geq 0} H^0(\mathscr{X}_\Q, \mathscr{L}_\Q^{\otimes k})
\]
is a finitely generated quasi-filtered algebra over $\Q$ (in the sense of definition 3.2.1 of \cite{Chen-1}). The convergence result for the measures then follows from a more general result of Chen (theorem 3.4.3 \cite{Chen-1}) concerning the distribution of successive maxima for such algebras, which shows, in particular, that the limiting distribution has compact support. 

In our setting, and that of \cite{TGS}, the Petersson metric on the line bundle $\mathscr{L}$ of modular forms is logarithmically singular, and hence to make sense of the norms at the archimedean places, we restricted to the subspace $H^0(\mathscr{X}_\Q, \mathscr{L}_\Q^{\otimes k}(-D))$ of sections vanishing along the cusps. The graded $\Q$-algebra
\[
B := \bigoplus_{k \geq 0} H^0(\mathscr{X}_\Q, \mathscr{L}_\Q^{\otimes k}(-D))
\]
is not finitely generated. However, we can ``approximate'' $B$ by a sequence of finite-type quasi-filtered $\Q$ algebras $B_L$ indexed by integers $L$ (as in \S \ref{cusp-forms-vanishing-higher-orders-ss}, and \S 3.6 of \cite{TGS}). Then applying theorem 3.4.3 of \cite{Chen-1} to each $B_L$, we get a sequence of compactly supported measures $\nu_{L, \infty}$, where the support of each $\nu_L$ is bounded above by a constant independent of $L$, but the lower bound for the support goes to $-\infty$ as $L \to \infty$, at least in the case $\Gamma = \Gamma(1)$ considered in \cite{TGS}. We then show that the $\nu_L$ converge to $\nu$ as $L \to \infty$, which accounts for the support of $\nu$ being unbounded below.
\end{Rem}

\begin{Rem}
A natural question is whether theorem \ref{main-theorem-statement-intro} generalizes to automorphic forms for other groups. More generally, we may ask the following general problem in the setting of \cite{Berman-Montplet}. Let $\mathscr{X}$ be an arithmetic variety with regular generic fiber, $\ov{\mathscr{L}}$ a hermitian line bundle on $\mathscr{X}$ with logarithmic singularities along a normal crossings divisor $D \in \text{Div}(\mathscr{X}_\Q)$. Fix a smooth volume form on $\mathscr{X}(\C)$ invariant under complex conjugation. For each $k \geq 1$, let $\mathscr{S}_k := H^0(\mathscr{X}, \mathscr{L}^{\otimes k}) \cap H^0(\mathscr{X}_\Q, \mathscr{L}_\Q^{\otimes k}(-D))$ denote the euclidean lattice of sections of $\mathscr{L}^{\otimes k}$ that vanish along $D$ equipped with the $L^2$ norm for the given volume form. Then do the probability measures $\nu_k$ associated to the (appropriately) normalized successive maxima of $\mathscr{S}_k$ converge weakly to a probability measure on $\R$? 

The case $\dim \mathscr{X} = 2$ (i.e., an arithmetic surface) is still interesting. One difficulty, even for $\dim \mathscr{X} = 2$, is the lack of a Gromov's lemma type result for general logarithmically singular line bundles, though some results are known with specific hypotheses on $\ov{\mathscr{L}}$ (see, for instance, \cite{AMM-Bergman-Kernels}). We hope to come back to this problem in the future. 
\end{Rem}

\vspace{0.5em}
\textbf{Outline of the paper.}
In \S \ref{background-section}, we start with some background on complex modular forms and modular curves. Then, following \cite{Kuhn}, we define integral models for modular curves and the line bundle of modular forms in \S \ref{integral-models-ss}, and describe the Petersson metric on this bundle in \S \ref{rational-Pet-met-ss}. We then introduce the notion of adelic vector bundles, as in definition 3.1 of \cite{Gaudron}, their successive maxima, and apply these concepts to the space of rational cusp forms of weight $12k$ for finite index subgroups of $\text{PSL}_2(\Z)$. This lets us define a decreasing $\R$-filtration on each $S_k$, which enables us to use results of Chen on quasi-filtered algebras in order to prove our results. Then in \S \ref{proof-of-thm-1-section} and \S \ref{measure-comp-proof-section}, we prove theorems \ref{main-theorem-statement-intro} and \ref{measure-comparison-thm-intro}, respectively. 

\vspace{1em}
\textbf{Acknowledgements.} I would like to thank my advisor, Ted Chinburg, for many helpful discussions regarding this project. 

\section{Background on integral modular forms}
\label{background-section}

In this section, we start with some background on complex modular forms and modular curves. Then, following \cite{Kuhn}, we define integral models for modular curves and the line bundle of modular forms in \S \ref{integral-models-ss}, and describe the Petersson metric on this bundle in \S \ref{rational-Pet-met-ss}. We then introduce the notion of adelic vector bundles, as in definition 3.1 of \cite{Gaudron}, their successive maxima, and apply these concepts to the space of rational cusp forms of weight $12k$ for finite index subgroups of $\text{PSL}_2(\Z)$.  

\subsection{Modular curves and modular forms over $\C$}
\label{mod-curves-over-C-ss}
Let $\Gamma$ be a finite index subgroup of $\Gamma(1) := \text{PSL}_2(\Z)$. Then $\Gamma$ acts on the (complex) extended upper half plane $\mf{h}^* := \mf{h} \cup \P^1(\Q)$ by linear fractional transformations
\[
\Gamma \ni \begin{pmatrix} a & b \\ c & d \end{pmatrix}  \cdot z = \frac{az + b}{cz + d}. 
\]
The quotient 
\[
X(\Gamma)_\C := \Gamma \backslash (\mf{h} \cup \P^1(\Q)). 
\]
is a compact Riemann surface called the \textit{modular curve associated to $\Gamma$}. The inclusion $\Gamma \subseteq \Gamma(1)$ induces a holomorphic map $\pi_{\Gamma, \C} : X(\Gamma)_\C \to X(\Gamma(1))_\C$ of Riemann surfaces of degree $[\Gamma(1) : \Gamma]$. 

For a point $z \in \mf{h}^*$, let 
\[
\Gamma_z := \{ \tau \in \Gamma : \tau \cdot z = z \}
\]
denote the stabilizer of $z$ in $\Gamma$. This group is either trivial, finite, or infinite cyclic, and we call the image of $z$ in $X(\Gamma)_\C$ an ordinary point, elliptic point, or cuspidal point (or just a cusp) in these cases, respectively. For $\Gamma(1)$, the elliptic points are the images of $i$ and $e^{2\pi i /3}$ and there is a unique cusp corresponding to the image of (any) point of $\P^1(\Q)$. For a general $\Gamma$, the elliptic points and cusps of $X(\Gamma)_\C$ map to elliptic points and cusp of $X(\Gamma(1))_\C$, respectively. 

Given $\gamma \in \text{PSL}_2(\R)$, the \textit{weight $k$ slash operator} $(\cdot|_k \gamma)$ on holomorphic functions on $\mf{h}$ is given by
\[
f(z) \mapsto (f|_k \gamma)(z) := (cz + d)^{-k} f \left( \frac{az + b}{cz + d} \right)
\]
for $\gamma = \begin{pmatrix} a & b \\ c & d \end{pmatrix}$. For $\Gamma \subseteq \Gamma(1)$, a \textit{meromorphic modular form of level $\Gamma$ and weight $k$} is a meromorphic function $f : \mf{h} \to \C$ such that 
\begin{enumerate}[(a)]
\item $(f|_k \gamma)(z) = f(z)$ for all $\gamma \in \Gamma$, and 
\item $f$ is meromorphic at the cusps of $\Gamma$. This means that at each cusp, if $t$ denotes a local parameter, then $f$ admits a Fourier expansion of the form $f(t) = \sum_n a_n t^n$, with $a_n = 0$ for sufficiently small $n$. 
\end{enumerate}
Meromorphic modular forms of weight $0$ are called \textit{modular functions} - these are simply the rational functions of $X(\Gamma)_\C$. A \textit{modular form of weight $k$ for $\Gamma$} is a meromorphic modular form that is holomorphic everywhere, including at the cusps. A \textit{cusp form of weight $k$ for $\Gamma$} is a modular form that is zero at every cusp (i.e., the $a_0$ coefficient in the local Fourier expansion at every cusp is zero). The space of modular (resp. cusp) forms of level $\Gamma$ and weight $k$ is denoted by $M_k(\Gamma)_\C$ (resp. $S_k(\Gamma)_\C$) - these are finite dimensional vector spaces over $\C$. 

\begin{Ex}
\label{j-function-example}
\begin{enumerate}[(i)]
\item The \textit{classical $j$-invariant function} is a modular function for $\Gamma(1)$ given by
\[
j(z) = \frac{1}{q} + 744 + 196884q + \dots,
\]
where $q = e^{2\pi i z}$ is the local parameter at the unique cusp $S_\infty$ of $X(1)_\C := X(\Gamma(1))_\C$. Note that $j$ has a unique pole of order $1$ at $S_\infty$, so that the induced morphism $j : X(1)_\C \to \P_\C^1$ is an isomorphism. Henceforth, we identify $X(1)_\C$ with $\P_\C^1$ via $j$. 

\item The \textit{classical discriminant function} $\Delta$ is a cusp form of weight $12$ for $\Gamma(1)$ given by 
\[
\Delta(z) = q \prod_{n=1}^\infty (1 - q^n)^{24},
\]
where as before, $q = e^{2\pi i z}$. 
\end{enumerate}
\end{Ex}

The functional equations satisfied by meromorphic modular forms show that they are sections of a line bundle $\mf{M}_k(\Gamma)_\C$ on $X(\Gamma)_\C$, and the meromorphic modular forms that are zero at every cusp form a sub-bundle $\mf{S}_k(\Gamma)_\C$. The following is proposition 4.7 in \cite{Kuhn}:
\begin{Prop}
\label{expl-desc-mod-forms-prop}
Let $\Gamma \subseteq \Gamma(1)$ be a finite index subgroup, and let $k \geq 1$ be a positive integer. Let $S_1, \dots, S_t \in X(\Gamma)_{\C}$ denote the cusps, and let $\pi_{\Gamma, \C} : X(\Gamma)_{\C} \to \P_{\C}^1$ denote the canonical map. Let $D_\C := S_1 + \dots + S_t \in \emph{Div}(X(\Gamma)_\C)$. Then we have an isomorphism of line bundles: 
\[
\mf{M}_{12k}(\Gamma)_{\C} \to \pi_{\Gamma, \C}^* \mc{O}_{\P_\C^1}(\infty)^{\otimes k} : f \mapsto f/\Delta^k,
\]
which restricts to an isomorphism
\[
\mf{S}_{12k}(\Gamma)_{\C} \to \left(\pi_{\Gamma, \C}^* \mc{O}_{\P_\C^1}(\infty)^{\otimes k}\right)(-D_\C). 
\] 
\end{Prop}

\subsection{Petersson metric on $\mf{M}_{12k}(\Gamma)$} 
\label{cx-Pet-metric-ss}
The \textit{Petersson metric} on $\mf{M}_{12k}(\Gamma)$ is defined as follows. Given a section $f$ of $\mf{M}_{12k}(\Gamma)$ over the open subset $U \subseteq X(\Gamma)_\C$ and a point $z \in \mf{h}$ corresponding to a point of $U$, we set
\[
|f|_{\text{Pet}}(z)^2 := |f(z)|^2 (4\pi \im(z))^{12k}. 
\]
This differs from the classical Petersson metric by a factor of $(4\pi)^{12k}$. See \cite{Kuhn} p. 227-228 for the reason for this factor. This metric has \textit{logarithmic singularities} in the sense of definition 3.1 of \cite{Kuhn} at the elliptic points and cusps of $\Gamma$ (see proposition 4.9 in \cite{Kuhn}). 

Let $d_\Gamma := [\Gamma(1) : \Gamma]$, and let $\mc{F}_\Gamma$ denote a fundamental domain for the action of $\Gamma$ on $\mf{h}$. Define the \textit{Petersson inner product} on $S_{12k}(\Gamma)_\C$ by 
\[
\lan f, g\ran_{\text{Pet}} := \frac{1}{d_\Gamma} \int_{\mc{F}_\Gamma} f(z) \ov{g(z)} (4 \pi \im(z))^{12k} \ \frac{dx dy}{y^2} \qquad (f,g \in S_{12k}(\Gamma)_\C),
\]
where $z = x + iy \in \mf{h}$. This is a Hermitian inner product on $S_{12k}(\Gamma)_\C$, and the norm of $f \in S_{12k}(\Gamma)_\C$ is given by $\|f\|_{\text{Pet}} := \lan f, f \ran_{\text{Pet}}^{1/2}$, which we note is simply the $L^2$ norm of the Petersson metric (with respect to the hyperbolic volume form on $X(\Gamma)_\C$). 

\subsection{Integral models of modular curves and modular forms}
\label{integral-models-ss}
We follow the setup in \S 4.11 of \cite{Kuhn} (with slightly modified notation). Let $\Gamma \subseteq \Gamma(1)$ be a finite index subgroup, and let $\pi_{\Gamma, \C} : X(\Gamma)_{\C} \to \P^1_\C$ denote the natural map given by the $j$-function. The branch points of $\pi_{\Gamma, \C}$ are contained in $\{0, 1728, \infty\} \subseteq \P^1_\C$, and hence $X(\Gamma)_\C$ and $\pi_{\Gamma, \C}$ are defined over a number field $E$. For such an $E$, let $X(\Gamma)_E$ be a smooth projective geometrically connected curve over $E$ with base change to $\C$ (via the implicit embedding of $E$ into $\C$) isomorphic to $X(\Gamma)_\C$, and let $\pi_{\Gamma, E} : X(\Gamma)_E \to \P_E^1$ be an $E$-morphism such that its base change to $\C$ (again, via the implicit embedding) coincides with $\pi_{\Gamma, \C}$. To simplify notation, we will refer to $X(\Gamma)_E$ by $X(\Gamma)$. Let $\pi_\Gamma : X(\Gamma) \to \P_\Q^1$ denote the composition of $\pi_{\Gamma, E}$ with the natural map $\P_E^1 \to \P_\Q^1$.  

Let $X(\Gamma)_\Z$ denote the normalization of $\P^1_\Z$ in $X(\Gamma)$ under the natural map $X(\Gamma) \to \P_\Q^1 \to \P_\Z^1$. Then $X(\Gamma)_\Z$ is a normal arithmetic surface with $X(\Gamma)$ as generic fiber. By results of Lipman (see \cite{Lipman}), there exists a \textit{desingularization} of $X(\Gamma)_\Z$. Namely, there exists a regular projective arithmetic surface with a proper birational morphism to $X(\Gamma)_\Z$. Let $\mathscr{X}(\Gamma) \to X(\Gamma)_\Z$ denote such a desingularization, and let the composite map $\mathscr{X}(\Gamma) \to X(\Gamma)_\Z \to \P_\Z^1$ be denoted by $\pi_{\Gamma, \Z} : \mathscr{X}(\Gamma) \to \P_\Z^1$. As in \S 4.11 in \cite{Kuhn}, we call any such $\mathscr{X}(\Gamma)$ an \textit{arithmetic surface associated to $\Gamma$}. We remark that this definition differs slightly from that in \S 4.11 of \cite{Kuhn} in that we do not require the field of constants $E$ of the generic fiber $X(\Gamma)$ of $\mathscr{X}(\Gamma)$ to be of minimal degree here.

Let $\infty \in \P_\Q^1$ correspond to the unique pole of the $j$-function, and let $\ov{\infty} \subseteq \P_\Z^1$ be its Zariski closure. Let $\mathscr{L} := \pi_{\Gamma, \Z}^* \mc{O}_{\P^1_\Z}(\ov{\infty})$. 

Now, 
\[
\mathscr{X}(\Gamma)_\C := \mathscr{X}(\Gamma) \otimes_\Z \C = \bigsqcup_{\sigma : E \to \C} X(\Gamma) \otimes_{E, \sigma} \C,
\]
where the disjoint union is over all embeddings $\sigma$ of $E$ into $\C$. For each $\sigma$, the base change $X(\Gamma) \otimes_{E, \sigma} \C$ is also a modular curve, say associated to the group $\Gamma_\sigma \subseteq \Gamma(1)$. Then 
\[
\mathscr{X}(\Gamma)_\C \cong \bigsqcup_{\sigma : E \to \C} X(\Gamma_\sigma)_\C. 
\]
%If $S_{\sigma, i}$ denote the pre-images of the $T_j$ under the base change map $X(\Gamma_\sigma)_\C \to X(\Gamma)$, and $e_{\sigma, i}$ is the ramification index of $S_{\sigma, i}$ over $\infty \in P_\C^1$ under the natural map $X(\Gamma_\sigma)_\C \to \P_\C^1$, then 
For each $k \geq 1$, 
\[
\mathscr{L}^{\otimes k}_\C = \bigoplus_{\sigma : E \to \C} \pi_{\Gamma_\sigma, \C}^*\mc{O}_{\P_\C^1}(\infty)^{\otimes k}
\]
is identified with the sum of the line bundles of modular forms of level $\Gamma_\sigma$ and weight $12k$ by proposition \ref{expl-desc-mod-forms-prop}. Hence, we call $\mathscr{L}$ the \textit{line bundle on $\mathscr{X}(\Gamma)$ associated to modular forms of level $\Gamma$ and weight $12$}. 

Now let $\{T_1, \dots, T_r\} \subseteq X(\Gamma)$ denote the preimage of $\infty \in \P_\Q^1$ under $\pi_\Gamma$, and let $D := T_1 + \dots + T_r \in \text{Div}(X(\Gamma))$ be the divisor of cusps of $X(\Gamma)$. Then, again by proposition \ref{expl-desc-mod-forms-prop}, 
\[
\mathscr{L}^{\otimes k}_\C(-D_\C)  := \mathscr{L}_\Q^{\otimes k}(-D) \otimes_\Q \C
\]
is the sum of the bundles of cusp forms for the $\Gamma_\sigma$. 

\begin{Rem} 
We will be interested in studying the submodule of $H^0(\mathscr{X}(\Gamma), \mathscr{L}^{\otimes k})$ associated to cusp forms (see \S \ref{ad-Q-vb-rat-cusp-forms-ss}). As such, it does not matter which desingularization $\mathscr{X}(\Gamma) \to X(\Gamma)_\Z$ we choose, since for any such desingularization, and for any vector bundle $E$ on $X(\Gamma)_\Z$, if $\mathscr{E}$ denotes its pullback to $\mathscr{X}(\Gamma)$, then the natural map $H^0(X(\Gamma)_\Z, E) \to H^0(\mathscr{X}(\Gamma), \mathscr{E})$ is an isomorphism.
\end{Rem}

\subsection{Petersson metric on $\mathscr{L}^{\otimes k}$}
\label{rational-Pet-met-ss}
For $k \geq 1$, we endow $\mathscr{L}^{\otimes k}$ with the \textit{Petersson metric} as described in \S \ref{cx-Pet-metric-ss}. Explicitly: for a section $f = (f_\sigma)_\sigma$ of $\mathscr{L}^{\otimes k}_\C$ over the open $\bigsqcup_\sigma U_{\sigma} \subseteq \bigsqcup_\sigma X(\Gamma_\sigma)_{\C}$,
\[
|f|_{\text{Pet}}^2(z) := |(f_\sigma \Delta^{k}) (z)|^2 (4\pi \im(z))^{12k},
\]
where $z \in \mf{h}^*$ corresponds to a point of $U_\sigma \subseteq X(\Gamma_\sigma)_{\C}$. We remark that even though there are choices involved in picking the groups $\Gamma_\sigma$ such that $X(\Gamma) \otimes_{E, \sigma} \C \cong X(\Gamma_\sigma)_\C$, the resulting metric on $\mathscr{L}$ is independent of these choices. As already mentioned in \S \ref{cx-Pet-metric-ss}, this metric is logarithmically singular at the elliptic points and cusps of the $\Gamma_\sigma$. We refer to $\mathscr{L}^{\otimes k}$ endowed with the Petersson metric by $\ov{\mathscr{L}}^{\otimes k}$.

%If $f$ belongs to $\mathscr{L}_\C^{\otimes k}(-D_\C)$, however, then $|f|_{\text{Pet}}^2(z) = 0$ for $z$ corresponding to any cusp, and in fact the function 
%\[
%|f|_{\text{Pet}}^2 : \sqcup_\sigma U_\sigma \to \R
%\]
%is continuous. 

For each $\sigma$, denote by $d\mu_\sigma$ the invariant measure on $X(\Gamma_\sigma)_{\C}$ induced by $\frac{dx dy}{y^2}$ on the upper half plane, where we use the coordinate $z = x + i y$, and let $d\mu$ denote the measure $\bigsqcup_\sigma d\mu_\sigma$ on $\bigsqcup_\sigma X(\Gamma_\sigma)_\C$. 

For $f \in H^0(\mathscr{X}(\Gamma)_\C, \mathscr{L}_\C^{\otimes k}(-D_\C))$ a global section, define 
\[
\|f\|_{k,\infty}^2 := \frac{1}{[E : \Q] d_{\Gamma}} \int_{\mathscr{X}(\Gamma)(\C)} |f|_{\text{Pet}}^2 \ d\mu
\]
to be the normalized $L^2$-norm of the Petersson metric. Note that for $f = (f_\sigma)$, 
%%%%%%%%%
\iffalse
\begin{align*}
\|f\|_{k,\infty}^2 &= \frac{1}{[E : \Q] d_\Gamma} \int_{\mathscr{X}(\Gamma)(\C)} |f|_{\text{Pet}}^2 \ d\mu \\
%&= \frac{1}{[E : \Q]} \sum_{\sigma : E \to \C}\int_{X(\Gamma_\sigma)_\C(\C)} |f|_{\text{Pet}}^2(z_\sigma) \ d\mu_{\sigma}(z_\sigma) \\ 
&= \frac{1}{[E : \Q]} \sum_{\sigma} \frac{1}{d_{\Gamma_\sigma}}\int_{X(\Gamma_\sigma)(\C)} |f_\sigma\Delta^k(z)|^2 (4\pi y)^{12k} \ \frac{dx dy}{y^2} \\
&= \frac{1}{[E : \Q]} \sum_{\sigma} \| f_\sigma \Delta^k \|_{\text{Pet}}^2,
\end{align*}
\fi
%%%%%%%%%
\[
\|f\|_{k,\infty}^2 = \frac{1}{[E : \Q]} \sum_{\sigma} \| f_\sigma \Delta^k \|_{\text{Pet}}^2,
\]
where $\| \cdot \|_{\text{Pet}}$ refers to the Petersson norm from \S \ref{cx-Pet-metric-ss}.  

\subsection{Adelic vector bundles, heights, and successive maxima} 
\label{adelic-vb-sm-ss}
Let $K$ be a number field, and let $\Sigma_K$ denote the set of all places of $K$. For $v \in \Sigma_K$, let $\C_v$ denote the completion of an algebraic closure of $K_v$. For $v$ finite, let $|\cdot|_v$ denote the norm on $\C_v$ that extends the $p$-adic norm on $\Q_p$, where $p$ is the rational prime lying under $v$. For $v$ an archimedean place, we let $| \cdot |_v$ denote the usual norm on $\C_v = \C$. 

A finite dimensional $K$-vector space $V$ is called an \textit{adelic vector bundle over $K$} if for each $v \in \Sigma_K$, $V \otimes_K \C_v$ is equipped with a norm $\| \cdot \|_v$ subject to the following conditions:
\begin{enumerate}[(i)]
\item There exists a $K$-basis $(s_1, \dots, s_r)$ of $V$ such that for all but finitely many finite places $v \in \Sigma_K$, 
\[
\| \alpha_1 s_1 + \dots + \alpha_r s_r \|_v = \max(|\alpha_1|_v, \dots, |\alpha_r|_v). 
\]
\item For every $v \in \Sigma_K$, $\| \cdot \|_v$ is invariant under the action of the group $\Gal(\C_v/K_v)$. Namely, if $(s_1, \dots, s_r)$ is a $K_v$-basis of $E \otimes_K K_v$, and if $\alpha_1, \dots, \alpha_r \in \C_v$, then 
\[
\| \tau(\alpha_1) s_1 + \dots + \tau(\alpha_r) s_r \|_v = \| \alpha_1 s_1 + \dots + \alpha_r s_r \|_v
\]
for all $\tau \in \Gal(\C_v/K_v)$. 
\item For $v \in \Sigma_K$ finite, $\| \cdot \|_v$ satisfies $\| s + s' \|_v \leq \max( \| s \|_v, \| s' \|_v )$. 
\end{enumerate}

Let $V$ be an adelic vector bundle over $K$. The \textit{naive adelic height} function on $V$ is given by
\begin{align*}
\lambda : V &\to \R \\
s &\mapsto -\sum_{v \in \Sigma_K} k_v \log \| s \|_v,
\end{align*}
where $k_v := [K_v : \Q_p]$ for $p$ the rational place lying under $v$. Using $\lambda$, we equip $V$ with the following filtration: given $a \in \R$, set 
\[
V^a := \text{span}_K\{ s \in V : \lambda(s) \geq a \}. 
\]
This is a decreasing filtration on $V$ indexed by the real numbers with the property that $V^a = V$ for $a \ll 0$, and $V^a = 0$ for $a \gg 0$. If $\dim_K V \geq 1$, the \textit{naive adelic successive maxima} of $V$ are the real numbers $\lambda_1, \dots, \lambda_{\dim_K V}$, where 
\[
\lambda_{i} := \sup \{ a \in \R : \dim_K V^a \geq i \}. 
\]

\begin{Ex} 
\label{ad-Q-vb-example}
Suppose $\mathscr{V}$ is a finite, free $\Z$-module of rank $d \geq 1$, equipped with a norm $\| \cdot \|_\infty$ on the $\C$-vector space $\mathscr{V} \otimes_\Z \C$. Then $V := \mathscr{V} \otimes_\Z \Q$ has a natural structure of an adelic vector bundle over $\Q$ as follows. For $v = \infty$, we use the given norm $\| \cdot \|_\infty$ on $V \otimes_\Q \C$, and for $v = p$ finite, define $\| \cdot \|_p : V \otimes_\Q \C_p \to \R$ by
\begin{equation}
\label{local-norm-finite-place}
\| s \|_p := \inf_{\alpha \in \C_p} \{ |\alpha|_p : s \in \alpha \left( \mathscr{V} \otimes_\Z R_p \right) \},
\end{equation}
where $R_p$ is the closed unit ball of $\C_p$. If $s_1, \dots, s_d$ denotes any $\Z$-basis of $\mathscr{V}$, then if we express $s \in V \otimes_\Q \C_p$ as $s = \sum_i \alpha_i s_i$ with $\alpha_i \in \C_p$, we have $\| s \|_p = \max_i (|\alpha_i|_p)$. 

If $\lambda_1, \dots, \lambda_d$ denote the successive maxima for $V$ with respect to the filtration induced by the naive adelic height as above, then 
\[
\lambda_i = -\log(\mu_i),
\]
where $\mu_1, \dots, \mu_d$ are the successive minima for the lattice $\mathscr{V} \subseteq \mathscr{V} \otimes_\Z \R$, where $\mathscr{V} \otimes_\Z \R$ is equipped with the norm induced from that on $\mathscr{V} \otimes_\Z \C$. 
\end{Ex}

\subsection{Adelic $\Q$-vector bundle structure on rational cusp forms}
\label{ad-Q-vb-rat-cusp-forms-ss}
Let $\Gamma \subseteq \Gamma(1)$ be a finite index subgroup and let $X(\Gamma), \mathscr{X}(\Gamma), \mathscr{L}$, and $D$ be as in \S \ref{integral-models-ss}. Let 
\begin{align*}
\mathscr{M}_k &:= H^0(\mathscr{X}(\Gamma), \mathscr{L}^{\otimes k}), \\
M_k &:= \mathscr{M}_k \otimes_\Z \Q, \\
S_k &:= H^0(X(\Gamma), \mathscr{L}_\Q^{\otimes k}(-D)),  \\
\mathscr{S}_k &:= \mathscr{M}_k \cap S_k,
\end{align*}
where the last intersection takes place in $M_k$. These correspond to the spaces of integral modular forms, rational modular forms, rational cusp forms, and integral cusp forms of level $\Gamma$ and weight $12k$, respectively. 

Note that $\mathscr{S}_k$ is a finite free $\Z$-module with $S_k = \mathscr{S}_k \otimes_\Z \Q$. By example \ref{ad-Q-vb-example}, $S_k$ has a natural structure of an adelic vector bundle over $\Q$, with respect to the norms $\| \cdot \|_{k,p}$, as in equation \ref{local-norm-finite-place} for a finite place $p$, and $\| \cdot \|_{k, \infty}$, the normalized $L^2$ norm from \S \ref{rational-Pet-met-ss}. We denote the naive adelic height function on $S_k$ by $\lambda_k$, the induced filtration by $(S_k^a)_{a \in \R}$, and the associated successive maxima by $\lambda_{k,i}$ (for $i = 1, \dots, \dim_\Q S_k$) (as in \S \ref{adelic-vb-sm-ss}). These maxima are equal to the ones in theorem \ref{main-theorem-statement-intro} by example \ref{ad-Q-vb-example}.

\section{Proof of theorem \ref{main-theorem-statement-intro}}
\label{proof-of-thm-1-section}
A key ingredient in the proof of theorem \ref{main-theorem-statement-intro} is theorem 3.4.3 in \cite{Chen-1} concerning quasi-filtered graded algebras. We apply this result to the graded algebras $B_L$ in proposition \ref{nu-L-conv-prop}. To do so, we need to get a uniform upper bound on the normalized height $\lambda_k(f)/k$ of any non-zero $f \in S_k$, and show that the algebra $B_L$ is \textit{quasi-filtered} with respect to an appropriate function (see definition 3.2.1 of \cite{Chen-1}). 

The uniform upper bound is proved in \S \ref{upper-bound-heights-ss}. The key ingredients for this are formulas for the intersection numbers of logarithmically singular line bundles as in \cite{Kuhn} and heights of cycles with respect to a smoothly metrized line bundle as in \cite{Bost-Gillet-Soule} (see proposition \ref{height-upper-bound-prop}), along with a version of ``Gromov's lemma'' proved in \cite{FJK} for cusp forms with respect to the Petersson norm and later improved (and generalized) in \cite{AMM-Bergman-Kernels} (see proposition \ref{Gromov-lemma-prop}). Proposition \ref{Gromov-lemma-prop} is also key to lemma \ref{quasi-filtered-lemma}, which in turn is used to show that the $B_L$ are quasi-filtered. Using all of this, we show that the $\nu_k$ converge vaguely to a sub-probability Borel measure $\nu$ on $\R$. 

Finally, in \S \ref{prob-measure-proof-ss}, we derive explicit lower bounds on the successive maxima $\lambda_{k,i}$ for $k$ large by constructing a basis for $S_k$ of a particular shape, and doing explicit calculations with it. These lower bounds then imply that $\nu$ is a probability measure (and hence the vague convergence is also weak convergence). 

We follow the general structure of \S 3.6 and \S 3.7 of \cite{TGS} in \S \ref{cusp-forms-vanishing-higher-orders-ss} and \S \ref{prob-measure-proof-ss}. 

Throughout this section, we keep the notation from \S \ref{ad-Q-vb-rat-cusp-forms-ss}. 

\subsection{Upper bound on heights}
\label{upper-bound-heights-ss}

\begin{Lemma}
\label{basic-scaling-lemma}
Let $f \in S_k$ be a non-zero element. There exists a rational number $\beta$ such that $\beta f \in \mathscr{S}_k$ and $\beta f \not\in B \cdot \mathscr{S}_k$ for all positive integers $B \geq 2$. Furthermore, for any $f \in \mathscr{S}_k$ such that $f \not\in B \cdot \mathscr{S}_k$ for all integers $B > 1$, we have $\| f \|_{k,p} = 1$ for all finite places $p$, and hence, $\lambda_k(f) = -\log \|f\|_{k, \infty}$. 
\end{Lemma}
\begin{proof}
Take any $\Z$-basis $f_1, \dots, f_d$ of $\mathscr{S}_k$, and let $f = \sum_i \alpha_i f_i$ with $\alpha_i \in \Q$. Then it is clear that there is a rational number $\beta$ such that $\beta \alpha_i \in \Z$ for all $i$ and $\gcd(\beta \alpha_1, \dots, \beta \alpha_d) = 1$. It is also clear that $\beta f \in \mathscr{S}_k$, and $\beta f \not\in B \cdot \mathscr{S}_k$ for all integers $B > 1$. 

Finally given any $f \in \mathscr{S}_k$ such that $f \not\in B \cdot \mathscr{S}_k$ for all integers $B > 1$, writing $f = \sum_i \alpha_i f_i$, we have $\alpha_i \in \Z$ and $\gcd(\alpha_1, \dots, \alpha_d) = 1$. By example \ref{ad-Q-vb-example}, we conclude that $\|f\|_{k,p} = \max_i (|\alpha_i|_p) = 1$ for all finite places $p$. 
\end{proof}

Let $\Gamma \subseteq \Gamma(1)$ be a finite index subgroup, and let $S_{2k}(\Gamma)_\C$ denote the space of complex cusp forms for $\Gamma$ of weight $2k$. Let $\lan \cdot, \cdot \ran_\text{Pet}$ denote the Petersson inner product on $S_{2k}(\Gamma)_\C$ as in \S \ref{cx-Pet-metric-ss}. For $f \in S_{2k}(\Gamma)$, denote by
\[
\| f \|_{\sup} := \sup_{z \in \Gamma \backslash \mf{h}} |f(z)| (4\pi \im(z))^{k}
\]
the sup-norm of $f$ with respect to the Petersson metric. 

\begin{Prop}
\label{Gromov-lemma-prop}
Let $\Gamma \subseteq \Gamma(1)$ be a finite index subgroup, and let $S_{2k}(\Gamma)_\C$ denote the space of complex cusp forms for $\Gamma$ of weight $2k$. There exist positive constants $c_1$ and $c_2$, with $c_2$ independent of $\Gamma$, such that for any $0 \neq f \in S_{2k}(\Gamma)_\C$, we have
\[
c_1 \|f\|_{\emph{Pet}} \leq \|f\|_{\sup} \leq c_2 k^{3/4} \|f\|_{\emph{Pet}}. 
\]
\end{Prop}
\begin{proof}
Let $\mu_\Gamma$ be the measure on $X(\Gamma)_\C$ induced from the hyperbolic volume form $\frac{dx dy}{y^2}$ on $\mf{h}$ with coordinate $z = x + iy$. For the first inequality, note that
\[
\|f\|_{\text{Pet}}^2 = \frac{1}{d_\Gamma} \int_{X(\Gamma)_\C} |f(z)|^2 (4\pi \im(z))^{2k} \ d\mu_\Gamma(z) \leq \|f\|_{\sup}^2 \left( \frac{1}{d_\Gamma} \int_{X(\Gamma)_\C} d\mu_\Gamma \right),
\]
so we may let $c_1 = (\frac{1}{d_\Gamma}\int_{X(\Gamma)_\C} d\mu_\Gamma)^{-1/2}$. 

For the other direction, let $\{f_1, \dots, f_d\}$ be an orthonormal basis for $S_{2k}(\Gamma)_\C$ for the Petersson inner product. Note that by our normalization of the Petersson inner product (i.e. the inclusion of the $(4\pi)^{2k}$ factor), $\{ (4\pi)^k f_1, \dots, (4 \pi)^k f_d \}$ is an orthonormal basis for the classical Petersson inner product. Then, for $z \in \mf{h}$, define
\[
B_{k}^{\Gamma}(z) := \sum_{j=1}^d |(4\pi)^k f_j(z)|^2 \im(z)^{2k} = \sum_{j=1}^d |f_j(z)|^2 (4 \pi \im(z))^{2k}
\]
as in \cite{AMM-Bergman-Kernels}. Note that for any $0 \neq f \in S_{2k}(\Gamma)_\C$, if $f = \sum_{j=1}^d \alpha_j f_j$, then $\|f\|_{\text{Pet}}^2 = \sum_{j=1}^d |\alpha_j|^2$, and for any $z \in \mf{h}$, 
\begin{align*}
|f(z)|^2 (4\pi\im(z))^{2k} &= \left| \sum_{j=1}^d \alpha_j f_j(z) \right|^2 (4\pi \im(z))^{2k} \\
&\leq \left( \sum_{j=1}^d |\alpha_j|^2 \right) \left(\sum_{j=1}^d |f_j(z)|^2 \right) (4\pi \im(z))^{2k}. 
\end{align*}
Hence, 
\[
\frac{|f(z)|^2 (4\pi \im(z))^{2k}}{\|f\|_{\text{Pet}}^2} \leq B_{k}^\Gamma(z),
\]
and 
\[
\frac{\|f\|_{\sup}^2}{\|f\|_{\text{Pet}}^2} = \sup_{z \in \Gamma \backslash \mf{h}} \frac{|f(z)|^2 (4\pi \im(z))^{2k}}{\|f\|_{\text{Pet}}^2} \leq \sup_{z \in \Gamma \backslash \mf{h}}  B_{k}^\Gamma(z). 
\]
By theorem 1.7 in \cite{AMM-Bergman-Kernels},
\[
\sup_{z \in \Gamma \backslash \mf{h}} B_{k}^\Gamma (z) = \left(\frac{k}{\pi}\right)^{3/2} + O(k). 
\]
In particular, there exists a constant $c_2$ (independent of $k$ and $\Gamma$) which gives us the desired bound. 
\end{proof}

We follow very closely the proof of lemma 4.1.7 in \cite{Chen-1} for the following result. 
\begin{Prop}
\label{height-upper-bound-prop}
There is a constant $C$ such that for any $0 \neq f \in S_k$, $\lambda_k(f) \leq Ck$. 
\end{Prop}
\begin{proof}
Since $\lambda_k(f) = \lambda_k(\alpha f)$ for any non-zero $\alpha \in \Q^\times$, we may suppose, after appropriately scaling $f$, that $f \in \mathscr{S}_k$ and $f \not\in B \cdot \mathscr{S}_k$ for any positive integer $B > 1$. Then by lemma \ref{basic-scaling-lemma}, $\lambda_k(f) = -\log \|f\|_{k, \infty}$. 

Taking any projective (closed) embedding $\Phi_\Z : \mathscr{X}(\Gamma) \hookrightarrow \P^N_\Z$, we let $\ov{L} := \Phi_\Z^* \ov{\mc{O}_{\P^N_\Z}(1)}$, where $\ov{\mc{O}_{\P^N_\Z}(1)}$ refers to $\mc{O}_{\P^N_\Z}(1)$ endowed with the Fubini-Study metric. Then $\ov{L}$ is arithmetically ample with $c_1(\ov{L}) > 0$. Now take any global section $\ell \in H^0(\mathscr{X}(\Gamma), L)$ such that $\text{div}_{L}(\ell)$ and $\text{div}_{\mathscr{L}^{\otimes k}}(f)$ don't share any common horizontal divisors (i.e. their divisors on the generic fiber $X(\Gamma)$ have disjoint support). Then the generalized intersection number $\ov{L} \cdot \ov{\mathscr{L}}^{\otimes k}$ in the sense of equation (3.10) in \cite{Kuhn} is given by
\[
\ov{L} \cdot \ov{\mathscr{L}}^{\otimes k} = ( \ell \cdot f)_{\text{fin}} + \lan \ell \cdot f \ran_\infty,
\]
where $(\ell \cdot f)_{\text{fin}} = (\text{div}_{L}(\ell) \cdot \text{div}_{\mathscr{L}^{\otimes k}}(f))_{\text{fin}}$ is equal to  
\[
(\text{div}_{L}(\ell) \cdot \text{div}_{\mathscr{L}^{\otimes k}}(f))_{\text{fin}} := \sum_{i,j=0}^1 (-1)^{i+j} \log \# H^i\left( \mathscr{X}(\Gamma), \text{Tor}_j^{\mc{O}_{\mathscr{X}(\Gamma)}}( \mc{O}_{\text{div}_{L}(\ell)}, \mc{O}_{\text{div}_{\mathscr{L}^{\otimes k}}(f)}) \right)
\]
($\mc{O}_{D}$ denotes the structure sheaf for an effective divisor $D \subseteq \mathscr{X}(\Gamma)$), and $\lan \ell \cdot f \ran_\infty$ is given in our case by
\begin{align*}
\lan \ell \cdot f \ran_\infty = - \sum_{P \in \mathscr{X}(\Gamma)(\C)} n_P \log |\ell(P)|   - \int_{\mathscr{X}(\Gamma)(\C)} \log |f|_{\text{Pet}} \cdot c_1(\ov{L}),
\end{align*}
if $\text{div}_{\mathscr{L}^{\otimes k}_\C}(f) = \sum_{P \in \mathscr{X}(\Gamma)(\C)} n_P P$ (lemma 3.9 in \cite{Kuhn}). Note that this expression makes sense since $\text{div}_{\mathscr{L}^{\otimes k}_\C}(f)$ and $\text{div}_{L_\C}(\ell)$ have disjoint support. Then, using bilinearity of the intersection pairing, we have
\[
(\ell \cdot f)_{\text{fin}} - \sum_{P \in \mathscr{X}(\Gamma)(\C)} n_P \log |\ell(P)| = k(\ov{L} \cdot \ov{\mathscr{L}}) + \int_{\mathscr{X}(\Gamma)(\C)} \log |f|_{\text{Pet}} \cdot c_1(\ov{L}),
\]
where the expression on the left is the height $h_{\ov{L}}(\text{div}_{\mathscr{L}^{\otimes k}}(f))$ of the cycle $\text{div}_{\mathscr{L}^{\otimes k}}(f)$ with respect to $\ov{L}$, as defined in 3.1.1 in \cite{Bost-Gillet-Soule} (see also \S 2.3.4 of \cite{Bost-Gillet-Soule}, in particular, equation 2.3.17). Since $\ov{L}$ is arithmetically ample and $\text{div}_{\mathscr{L}^{\otimes k}}(f)$ is effective, $h_{\ov{L}}(\text{div}_{\mathscr{L}^{\otimes k}}(f)) \geq 0$ (proposition 3.2.4 in \cite{Bost-Gillet-Soule}). Hence 
\[
0 \leq k(\ov{L} \cdot \ov{\mathscr{L}}) + \int_{\mathscr{X}(\Gamma)(\C)} \log |f|_{\text{Pet}} \cdot c_1(\ov{L}),
\]
and since $c_1(\ov{L}) > 0$, we get
%\begin{align*}
%0 &\leq k(\ov{L} \cdot \ov{\mathscr{L}}) + \int_{\mathscr{X}(\Gamma)(\C)} \log |f|_{\text{Pet}} \cdot c_1(\ov{L}) \\
%&\leq k(\ov{L} \cdot \ov{\mathscr{L}}) + \log \sup_{z \in \mathscr{X}(\Gamma)(\C)} |f|_{\text{Pet}}(z) \int_{\mathscr{X}(\Gamma)(\C)}  c_1(\ov{L}),
%\end{align*}
%and hence 
\[
\log \sup_{z \in \mathscr{X}(\Gamma)(\C)} |f|_{\text{Pet}}(z) \geq C'k
\]
for $C' := -(\ov{L} \cdot \ov{\mathscr{L}})/\int_{\mathscr{X}(\Gamma)(\C)} c_1(\ov{L})$. 

Denote the image of $f$ under the map $H^0(\mathscr{X}(\Gamma), \mathscr{L}^{\otimes k}(-D)) \to H^0(\mathscr{X}(\Gamma)_\C, \mathscr{L}^{\otimes k}_\C(-D_\C))$ by $(f^\sigma)_\sigma$. Suppose $|f|_{\text{Pet}}(z)$ achieves its supremum on the component $X(\Gamma_{\sigma'})_\C$, so that 
\[
\sup_{z \in \mathscr{X}(\Gamma)(\C)} |f|_{\text{Pet}}(z) = \sup_{z \in X(\Gamma_{\sigma'})_\C} |(f^{\sigma'} \Delta^k)(z)| (4\pi \im(z))^{6k} = \|f^{\sigma'} \Delta^k \|_{\sup}. 
\]
Then
\[
\|f\|_{k, \infty}^2 = \frac{1}{[E : \Q]}\sum_{\sigma} \|f^\sigma \Delta^k \|_{\text{Pet}}^2 \geq \frac{1}{[E : \Q]} \|f^{\sigma'} \Delta^k\|_{\text{Pet}}^2 \geq \frac{1}{[E : \Q]} c_2^{-2} k^{-3/2} \|f^{\sigma'} \Delta^k \|_{\sup}^2,
\]
where the last inequality uses proposition \ref{Gromov-lemma-prop}. We conclude that 
\[
\lambda_k(f) = -\log \|f\|_{k, \infty} \leq Ck
\]
for a constant $C$ independent of $k$. 
\end{proof}

\iffalse
\subsection{Chen's theorem}
The key ingredient in the proof of theorem \ref{main-theorem-statement} is theorem 3.4.3 in \cite{Chen-1}, reproduced below in theorem \ref{Chen-thm}. First, we recall briefly the setup and terminology introduced in chapter 3 of \cite{Chen-1}. 

Let $B = \bigoplus_{n \geq 0} B_n$ be a graded algebra of finite type over a field $k$. Suppose that each $B_n$ is a finite dimensional $k$-vector space, equipped with a decreasing $\R$-filtration $(B_n^a)_{a \in \R}$ satisfying the hypotheses: $B_n^a = B_n$ for $a \ll 0$, $B_n^a = 0$ for $a \gg 0$, and for any real $b$, $B_n^b = \cap_{a < b} B_n^a$. 

\begin{Th}
\label{Chen-thm}
Let $B = \oplus_{n \geq 0} B_n$ be a graded algebra of finite type over a field $k$. Suppose that each $k$ vector space $B_n$ is finite dimensional and is equipped with a decreasing $\R$-filtration $\mathscr{F}^{(n)}$. Let $f : \Z_{\geq 0} \to \R_{\geq 0}$ be a function such that $\lim_{n \to \infty} f(n)/n = 0$. We suppose that:

(i) the graded algebra $B$ is integral and $f$-quasi-filtered, and $B_n \neq 0$ for $n$ large 

(ii) there exists $\alpha > 0$ such that $\lambda_{\max}(B_n, \mathscr{F}^{(n)}) \leq \alpha n$ for all $n \geq 1$. 

Then for every integer $n \geq 1$, let $\nu_n = T_{\frac{1}{n}} \nu_{\mathscr{F}^{(n)}}$. Then the supports of the measures $\nu_n$ are uniformly bounded and the sequence of measures $(\nu_n)_{n \geq 1}$ converges weakly to a compactly supported probability measure on $\R$. 
\end{Th}
\fi

\subsection{Cusp forms vanishing to increasing orders at the cusps}
\label{cusp-forms-vanishing-higher-orders-ss} 
For an integer $L \geq 1$, define 
\[
\mc{B}_{L,k} := \mathscr{L}_\Q^{\otimes k} \! \left( - \lceil k/L \rceil D \right)  \subseteq \mathscr{L}_\Q^{\otimes k}(-D). 
\] 
Then for each $\sigma : E \to \C$, 
\[
(\mc{B}_{L,k})_{\C}|_{X(\Gamma_\sigma)_\C} = \mathscr{L}_\C^{\otimes k} \! \left( -\lceil k/L \rceil) D_\C \right)|_{X(\Gamma_\sigma)_\C}
\]
is identified with the line bundle of weight $12k$ cusp forms for $X(\Gamma_\sigma)_\C$ that vanish to order at least $k/L$ at every cusp. 

Let $B_{L,k} := H^0(X(\Gamma), \mc{B}_{L,k})$. We use the inclusion $B_{L,k} \subseteq S_k$ to define a filtration on $B_{L,k}$. Namely, for $a \in \R$, we set the $a$th filtered piece of $B_{L,k}$ to be
\[
B_{L,k}^a := S_{k}^a \cap B_{L,k}. 
\]
This is a decreasing $\R$-filtration on $B_{L,k}$. For $i = 1, \dots, \dim_\Q B_{L,k}$, let 
\[
\lambda_{L,k,i} = \sup \{a \in \R : \dim_\Q B_{L,k}^a \geq i \}
\]
denote the $i$th successive maxima of $B_{L,k}$. Since these are the successive maxima associated to the subspace filtration, the multi-set $\{ \lambda_{L,k,i} \}_{i=1}^{\dim_\Q B_{L,k}}$ is a sub multi-set of $\{ \lambda_{k,i} \}_{i=1}^{\dim_\Q S_k}$. 

Furthermore, let
\[
\wt{\lambda}_k : S_k \to \R \cup \{\infty\}, \quad \wt{\lambda}_{L,k} : B_{L,k} \to \R \cup \{\infty\}
\]
be defined by $\wt{\lambda}_k(f) := \sup \{ a \in \R : f \in S_k^a \}$, and $\wt{\lambda}_{L,k}(f) := \sup \{a \in \R : f \in B_{L,k}^a \}$. These are called the \textit{index functions} of $S_k$ and $B_{L,k}$ for their respective filtrations (as in \S 2 of \cite{Chen-2}). For $f \in B_{L,k}$, we have $\wt{\lambda}_{L,k}(f) = \wt{\lambda}_k(f)$. 

\begin{Rem} 
We could have opted to use the filtration on $B_{L,k}$ obtained from the restriction of the height function $\lambda_k$ on $S_k$, so that for $a \in \R$ the $a$th filtered piece of $B_{L,k}$ would be 
\[
B_{L,k,a} := \text{span}_\Q \{ s \in B_{L,k} : \lambda_k(s) \geq a \}. 
\]
This approach is taken in \cite{TGS}, \S 3.6 (see, in particular, the proof of lemma 3.6.2) and makes proving various estimates (see \cite{TGS}, lemma 3.7.1) rather tricky. By contrast, working with the subspace filtration simplifies these matters significantly - compare lemma \ref{main-estimate-lemma} below to its counterpart, \cite{TGS}, lemma 3.7.1.
\end{Rem}

As mentioned at the start of this section, the following lemma is used to show the algebras $B_L$ in proposition \ref{nu-L-conv-prop} are quasi-filtered. 
\begin{Lemma}
\label{quasi-filtered-lemma}
Let $\psi(k) := \frac{3}{4}\log(k) + \log(c_2) - \log(c_1)/2$ where $c_1$ and $c_2$ are the constants from proposition \ref{Gromov-lemma-prop} for $\Gamma \subseteq \Gamma(1)$. For any collection of elements $f_i \in S_{k_i}$ ($i=1, \dots, n$, with $n \geq 2$), we have 
\[
\wt{\lambda}_{k_1 + \dots + k_n}(f_1 \cdots f_n) \geq \sum_{i=1}^n \left( \wt{\lambda}_{k_i}(f_i) - \psi(k_i) \right). 
\]
\end{Lemma}
\begin{proof}
Note that $c_1$ was set to be $(\int_{X(\Gamma)_\C} d\mu_\Gamma)^{-1}$, which only depends on the index of $\Gamma$ in $\Gamma(1)$. Since $[\Gamma(1) : \Gamma] = [\Gamma(1) : \Gamma_\sigma]$ for all $\sigma$, we may use the same $c_1$ for all the $\Gamma_\sigma$. Hence for any $f \in S_k$, proposition \ref{Gromov-lemma-prop} gives
\[
c_1^2 \|f\|_{k, \infty}^2 \leq \sum_\sigma \|f^\sigma \Delta^k \|_{\sup}^2 \leq c_2^2 k^{3/2} \|f\|_{k, \infty}^2. 
\]

Then for $f_i \in S_{k_i}$ ($i = 1, \dots, n$), letting $K := k_1 + \dots, k_n$, we get
\begin{align*}
c_1^2 \| f_1 \dots, f_n \|_{K, \infty}^2 &\leq \sum_\sigma \| (f_1 \dots f_n)^\sigma \Delta^K \|_{\sup}^2 \\
&\leq \sum_\sigma \|f_1^\sigma \Delta^{k_1} \|_{\sup}^2 \cdots \|f_n^\sigma \Delta^{k_n} \|_{\sup}^2 \\
&\leq \left( \sum_\sigma \|f_1^\sigma \Delta^{k_1} \|_{\sup}^2 \right) \cdots \left( \sum_\sigma \|f_n^\sigma \Delta^{k_n} \|_{\sup}^2 \right) \\
&\leq c_2^{2n} k_1^{3/2} \cdots k_n^{3/2} \|f_1\|_{k_1, \infty}^2 \dots \|f_n\|_{k_n, \infty}^2,
\end{align*}
and consequently,
\[
\| f_1 \dots f_n \|_{K, \infty} \leq e^{\psi(k_1) + \dots + \psi(k_n)} \|f_1\|_{k_1, \infty} \cdots \|f_n\|_{k_n, \infty}. 
\]
(We use the fact that $c_1 < 1$ here.) 

Pick any $\varepsilon > 0$. By definition, $f_i \in S_{k_i}^{\wt{\lambda}_{k_i}(f_i) - \varepsilon/n}$, and hence $f_i = \sum g_{i,j}$ for $g_{i,j} \in S_{k_i}$ with $\lambda_{k_i}(g_{i,j}) \geq \wt{\lambda}_{k_i}(f_i) - \varepsilon/n$. It is easy to see that for any finite place $p$, 
\[
\|g_{1, j_1} \cdots g_{n, j_n}\|_{K,p} \leq \|g_{1,j_1}\|_{k_1, p} \cdots \|g_{n,j_n}\|_{k_n,p}.
\]
This combined with the previous paragraph yields
\begin{align*}
\lambda_K(g_{1, j_1} \cdots g_{n, j_n}) &\geq \sum_{i=1}^n \left( \lambda_{k_i}(g_{i,j_i}) - \psi(k_i) \right) \\
&\geq \sum_{i=1}^n \left( \wt{\lambda}_{k_i}(f_i) - \psi(k_i) \right) - \varepsilon.
\end{align*}
Since $\varepsilon$ is arbitrary, we conclude that 
\[
\wt{\lambda}_K(f_1 \cdots f_n ) \geq \sum_{i=1}^n \left( \wt{\lambda}_{k_i}(f_i) - \psi(k_i) \right),
\]
as required. 
\end{proof}

Given a Borel measure $\nu$, and an integrable function $f$ on $\R$, let
\[
\nu(f) := \int_\R f \ d\mu. 
\]

\begin{Prop}
\label{nu-L-conv-prop}
With the notation above, the sequence of probability measures 
\[
\nu_{L,k} := \frac{1}{\dim_\Q B_{L,k}} \sum_{i=1}^{\dim_\Q B_{L,k}} \delta_{\frac{1}{k} \lambda_{L,k,i}}
\]
converges weakly as $k \to \infty$ to a probability measure $\nu_{L, \infty}$ with compact support. 
\end{Prop}
\begin{proof}
First, we consider the case $k = qL$ for integers $q \geq 1$. Let
\[
B_L := \bigoplus_{q \geq 0} B_{L,qL}.
\]
Since $\deg(\mc{B}_{L,L}) > 0$, $\mc{B}_{L,L}$ is ample, and hence $B_L$ is a finitely generated $\Q$-algebra with $B_{L,qL} \neq 0$ for $q$ large. We endow each $B_{L,qL}$ with the filtration $(B_{L,qL}^a)_{a \in \R}$ discussed above. Then by lemma \ref{quasi-filtered-lemma}, $B_L$ is $\psi_L$-quasi-filtered with $\psi_L(q) := \psi(qL)$, where $\psi$ is as in lemma \ref{quasi-filtered-lemma}. By proposition \ref{height-upper-bound-prop}, $\wt{\lambda}_{L, qL}(f) = \wt{\lambda}_{qL}(f) \leq CqL$ for all non-zero $f \in B_{L,qL}$. Hence by theorem 3.4.3 of \cite{Chen-1} (and by rescaling the measures by $1/L$), the collection of measures 
\[
\nu_{L,qL} := \frac{1}{\dim_\Q B_{L,qL}} \sum_{i=1}^{\dim_\Q B_{L,qL}} \delta_{\frac{1}{qL} \lambda_{L,qL,i}}
\]
converges weakly to a compactly supported probability measure $\nu_{L,\infty}$ on $\R$.

For general $k$, suppose that $k = qL + r$ with $0 \leq r < L$. For $f \in B_{L,k} \subseteq B_{L,(q+1)L}$, lemma \ref{quasi-filtered-lemma} gives 
\begin{align*}
\wt{\lambda}_{L,(q+1)L}(f) &\geq \wt{\lambda}_{L,k}(f) + \wt{\lambda}_{L,L-r}(1) -\psi(k) - \psi(L-r) \\
&\geq \wt{\lambda}_{L,k}(f) + c_3 \log(k),
\end{align*}
for some constant $c_3$ independent of $k$. This implies that for all $i = 1, \dots, \dim_\Q B_{L,k}$, 
\[
\lambda_{L,(q+1)L, i} \geq \lambda_{L,k,i} + c_3\log(k). 
\]
Then 
\[
\dim_\Q B_{L,(q+1)L} - \dim_\Q B_{L,k} = [E : \Q] d_\Gamma (L-r),
\]
is independent of $q$. Hence, for any bounded increasing continuous function $f : \R \to \R$, we have 
\[
\nu_{L,k}(f) \leq \nu_{L, (q+1)L}(f) + o(1),
\]
and hence,
\[
\limsup_{q \to \infty} \nu_{L, qL + r}(f) \leq \nu_{L, \infty}(f). 
\]
Similarly, using the inclusion $B_{L,qL} \subseteq B_{L,k}$, we deduce that 
\[
\liminf_{q \to \infty} \nu_{L, qL + r}(f) \geq \nu_{L,\infty}(f),
\]
and hence $\nu_{L,k}(f) \to \nu_{L,\infty}(f)$ for all bounded increasing continuous functions $f$. Since all measures involved are probability measures, this shows that $\nu_{L,k} \to \nu_{L, \infty}$ weakly. 
\end{proof}

\begin{Lemma}
\label{main-estimate-lemma}
For every positive real number $\varepsilon$ and bounded Lipschitz function $h : \R \to \R$ , there is a constant $L_0 = L_0(\varepsilon, h)$ such that for all $L \geq L_0$, and for all $k \geq 1$, we have
\[
|\nu_{k}(h) - \nu_{L,k}(h)| \leq \epsilon. 
\]
\end{Lemma}
\begin{proof}
Let $|h|_{\text{Lip}} := \sup_{x \in \R} |h(x)| + \sup_{x,y \in \R, x \neq y} \frac{|h(x) - h(y)|}{|x-y|} = M$. Recall that since $B_{L,k}$ is equipped with the subspace filtration coming from $S_k$, the multi-set $\{ \lambda_{L,k,i} \}_{i=1}^{\dim_\Q B_{L,k}}$ is a sub multi-set of $\{ \lambda_{k,i} \}_{i=1}^{\dim_\Q S_k}$. Let $d_{L,k} := \dim_\Q B_{L,k}$ and $d_k := \dim_\Q S_k$. Then
\begin{align*}
|\nu_k(h) - \nu_{L,k}(h)| &= \left| \frac{1}{d_k}\sum_{i=1}^{d_k} h\left( \frac{1}{k} \lambda_{k,i} \right) - \frac{1}{d_{L,k}} \sum_{i=1}^{d_{L,k}} h\left( \frac{1}{k} \lambda_{L,k,i} \right)  \right| \\
%&= \left| \frac{1}{d_k}\sum_{i=1}^{d_k} h\left( \frac{1}{k} \lambda_{k,i} \right) - \frac{1}{d_k} \sum_{i=1}^{d_{L,k}} h\left( \frac{1}{k} \lambda_{L,k,i} \right) + \frac{1}{d_k} \sum_{i=1}^{d_{L,k}} h\left( \frac{1}{k} \lambda_{L,k,i} \right) - \frac{1}{d_{L,k}} \sum_{i=1}^{d_{L,k}} h\left( \frac{1}{k} \lambda_{L,k,i} \right)  \right| \\
&\leq \frac{1}{d_k}\left|\sum_{i=1}^{d_k} h\left( \frac{1}{k} \lambda_{i,k} \right) - \sum_{i=1}^{d_{L,k}} h\left( \frac{1}{k} \lambda_{i,L,k} \right) \right| + \left(\frac{1}{d_{L,k}} - \frac{1}{d_k} \right) \left| \sum_{i=1}^{d_{L,k}} h\left( \frac{1}{k} \lambda_{i,L,k} \right) \right| \\
&\leq \frac{1}{d_k} (d_k - d_{L,k})M + \frac{d_k - d_{L,k}}{d_k d_{L,k}} d_{L,k} M \\
&\leq \frac{c_4 \left( \lceil k/L \rceil - 1\right)}{d_k} \\
&\leq \frac{c_4}{L},
\end{align*}
for some constant $c_4$. The conclusion follows at once. 
\end{proof}

\begin{Cor}
Suppose $N \geq 2$. Let $h : \R \to \R$ be a bounded Lipschitz function. Then the sequences $(\nu_k(h))_k$ and $(\nu_{L,\infty}(h))_L$ are convergent and have the same limit. 
\end{Cor}
\begin{proof}
Let $\varepsilon > 0$ be any positive real number, and let $L_0 = L_0(\varepsilon, h)$ be as in the previous lemma. For any $L \geq L_0$, and any $k$, we have $|\nu_k(h) - \nu_{L,k}(h)| \leq \varepsilon$ from the lemma above, and hence 
\[
\nu_{L,k}(h) - \varepsilon \leq \nu_k(h) \leq \nu_{L,k}(h) + \varepsilon. 
\]
From this we get 
\[
\nu_{L,\infty}(h) - \varepsilon \leq \liminf_k \nu_k(h) \leq \limsup_k \nu_k(h) \leq \nu_{L,\infty}(h) + \varepsilon,
\]
and hence 
\[
0 \leq \limsup_k \nu_k(h) - \liminf_k \nu_k(h) \leq 2\varepsilon.
\]
Since $\varepsilon$ is arbitrary, we conclude that $\lim_k \nu_k(h)$ exists. Moreover, since for all $L \geq L_0$, we have 
\[
|\nu_{L,\infty}(h) - \lim_k \nu_k(h)| \leq \varepsilon,
\]
we conclude that $\lim_L \nu_{L,\infty}(h) = \lim_k \nu_k(h)$, as required.
\end{proof}

Now by the Riesz representation theorem for measures, there exists a sub-probability Borel measure $\nu$ on $\R$ representing the positive linear functional $\lim_k \nu_k = \lim_L \nu_{L,\infty}$ on $C_c(\R)$. Namely, for all $h \in C_c(\R)$, we have 
\begin{equation}
\label{vague-conv-limit-eqn}
\nu(h) = \lim_k \nu_k(h) = \lim_L \nu_{L,\infty}(h). 
\end{equation}
Hence, $\nu_k$ converges \textit{vaguely} to $\nu$. 

\subsection{$\nu$ is a probability measure} 
\label{prob-measure-proof-ss}
We obtain lower bound estimates on the successive maxima $\lambda_{k,i}$ of $S_k$ in proposition \ref{lower-bounds-successive-max}, which are then used in proposition \ref{uniform-tightness-prop} to show that $\nu$ is a probability measure. 

Recall that $T_1, \dots, T_r$ are the pre-images of $\infty \in \P_\Q^1$ under the map $\pi_\Gamma : X(\Gamma) \to \P_\Q^1$, and $D = T_1 + \dots + T_r$. Let $e_i$ denote the ramification index of $T_i$ over $\infty$. Then
\[
\mathscr{L}_\Q^{\otimes k}(-D) = \mc{O}_{X(\Gamma)} \left( (ke_1 - 1)T_1 + \dots + (ke_r - 1)T_r \right). 
\]

Let $g$ denote the genus of $X(\Gamma)_\C$. Then the genus of $X(\Gamma)$ as a curve over $E$, and that of $X(\Gamma_\sigma)_\C$, for each $\sigma$, is also equal to $g$.  

\begin{Prop}
\label{lower-bounds-successive-max}
Fix an integer $k_0$ such that $k_0e_p - 1 > 2 g - 2$ for all $p = 1, \dots, r$, and let $d_0 := \dim_\Q S_{k_0}$. For all $k > k_0$, if $\{ \lambda_{k,i} \}_{i=1}^{\dim_\Q S_k}$ denote the successive maxima of $S_k$, then:
\begin{itemize}
\item For $1 \leq i \leq d_0$, 
\[
\lambda_{k,i} \geq 6k \log \left( \frac{k - k_0}{k} \right) - C_4 k.
\]
\item For $d_0 + 1 \leq i$, 
\[
\lambda_{i,k} \geq 6k \log \left( \frac{k - k_0 - \left\lceil \frac{i - d_0}{d_\Gamma [E : \Q]} \right\rceil}{k} \right) - C_4 k.
\]
\end{itemize}
For $i > d_0 + (k - k_0 - 1)d_\Gamma [E : \Q]$, we interpret the right hand side above as $-\infty$. 
\end{Prop}

\begin{proof}
We construct an $E$-basis for $S_k$ when $k \geq k_0$. We start by constructing a basis for $S_{k_0 + 1}/S_{k_0}$. Using proposition 1.40 in \cite{Shimura}, one easily checks that for $k \geq 1$, $\deg(\mathscr{L}_\Q^{\otimes k}(-D)) > 2g - 2$, and hence by Riemann-Roch,
\[
\dim_E S_k = \dim_E H^0(X(\Gamma), \mathscr{L}^{\otimes k}_\Q(-D)) = \deg (\mathscr{L}^{\otimes k}_\Q(-D)) - g + 1.
\]
for $k \geq 1$. Hence, for any $m \geq 0$, 
\begin{equation}
\label{codim-S_k}
\dim_E S_{k_0 + m + 1} - \dim_E S_{k_0 + m} = \deg \mathscr{L}_\Q = d_{\Gamma}. 
\end{equation}
For $p = 1, \dots, r$, and $q = 1, \dots, e_p$, since 
\[
\dim_E H^0( X(\Gamma), \mc{O}( (k_0e_p - 1 + q) T_p) ) - \dim_E H^0( X(\Gamma), \mc{O}( (k_0e_p - 1 + (q-1)) T_p) ) = \deg T_p,
\]
there exist rational functions $b_{q,p,i}$ for $i=1, \dots, \deg T_p$ in $H^0( X(\Gamma), \mc{O}( (k_0e_p - 1 + q) T_p) )$ that are $E$-linearly independent, and that are regular everywhere on $X(\Gamma)$ except at $T_p$, where they have a pole of order $k_0e_p - 1 + q$.

Varying $p,q$, and $i$, there are $d_{\Gamma}$ such functions $\{b_{q,p,i}\}_{q,p,i}$, and $b_{q,p,i} \in S_{k_0 + 1} \setminus S_{k_0}$. It is clear from construction that $\{ b_{q,p,i} \}_{q,p,i} \subseteq S_{k_0 + 1}$ are $E$-linearly independent, which in light of equation \ref{codim-S_k} implies that their images constitute an $E$-basis of $S_{k_0 + 1}/S_{k_0}$. 

For any $m \geq 0$, note that 
\[
\text{ord}_{T_p}(j^m b_{q,p,i}) = -me_p - (k_0e_p - 1 + q) = -[ (k_0 + m)e_p - 1 + q ]. 
\]

Since $\{ j^m b_{q,p,i} \}_{q,p,i} \subseteq S_{k_0 + m + 1} \setminus S_{k_0 + m}$ are $E$-linearly independent, equation \ref{codim-S_k} again implies that their images make up an $E$-basis of $S_{k_0 + m + 1} / S_{k_0 + m}$. Now let $d_0' = \dim_E S_{k_0}$ and fix an $E$-basis $\{ c_1, \dots, c_{d_0'} \}$ of $S_{k_0}$. For $k > k_0$, the above discussion yields the $E$-basis
\[
\{c_1, \dots, c_{d_0'} \} \cup \bigcup_{t=0}^{k - k_0 - 1} \{ j^t b_{q,p,i} \}_{q,p,i}
\]
of $S_k$. By scaling the $c_v$'s and the $b_{q,p,i}$'s by elements of $\Q^\times$, we may assume that they are integral (i.e., they lie in $\mathscr{S}_{k_0}$ and $\mathscr{S}_{k_0 + 1}$, respectively). Then since $j \in \mathscr{M}_1$ is integral, $j^t b_{q,p,i} \in \mathscr{S}_{k}$ are integral as well.  Fixing a $\Z$-basis $\{ \alpha_1, \dots, \alpha_{[E : \Q]} \}$ of the ring of integers $\mc{O}_E$, we get the $\Q$-basis of $S_k$:
\[
\bigcup_{u=1}^{[E : \Q]} \left( \{\alpha_u c_1, \dots, \alpha_u c_{d_0'} \} \cup \bigcup_{t=0}^{k - k_0 - 1} \{ \alpha_u j^t b_{q,p,i} \}_{q,p,i} \right). 
\]
Since each $\mathscr{S}_k$ is an $\mc{O}_E$-module, the $\alpha_u c_v$ and $\alpha_u j^t b_{q,p,i}$ are integral as well.  

Let $\alpha_u^\sigma c_v^\sigma$ and $\alpha_u^\sigma j^t b_{q,p,i}^\sigma$ denote the images of $\alpha_u c_v$ and $\alpha_v j^t b_{q,p,i}$, respectively, in $S_k \otimes_{E, \sigma} \C$. Each $\alpha_u^\sigma c_v^\sigma \Delta^{k_0} \in S_{12k_0}(\Gamma_\sigma)_\C$, and $\alpha_u^\sigma b_{q,p,i}^\sigma \Delta^{k_0 + 1} \in S_{12(k_0 + 1)}(\Gamma_\sigma)_\C$. There is a positive constant $C_1$ such that the estimates
\begin{align*}
\sup_{z_\sigma \in \Gamma_\sigma \backslash \mf{h}} |(\alpha_u^\sigma c_v^\sigma \Delta^{k_0}) (z)| &\leq C_1 \\
\sup_{z_\sigma \in \Gamma_\sigma \backslash \mf{h}} |(\alpha_u^\sigma b_{q,p,i}^\sigma \Delta^{k_0 + 1})(z)| &\leq C_1 \\
|j(z)| &\leq C_1 e^{2\pi y} \\
|\Delta(z)| &\leq C_1 e^{-2\pi y}
\end{align*}
hold for all $\sigma, u, v$, and triples $(q,p,i)$. Then 
\begin{align*}
\| \alpha_u^\sigma c_v^\sigma \Delta^k\|_{\text{Pet}}^2 &= \frac{1}{d_{\Gamma_\sigma}} \int_{\mc{F}_{\Gamma_\sigma}} |(\alpha_u^\sigma c_v^\sigma \Delta^k) (z)|^2 (4\pi y)^{12k} \ \frac{dx dy}{y^2} \\
%&= \frac{1}{d_{\Gamma_\sigma}} \int_{\mc{F}_{\Gamma_\sigma}} |(\alpha_u^\sigma c_v^\sigma \Delta^{k_0})(z_\sigma)|^2 |\Delta^{k - k_0} (z_\sigma)|^2 (4\pi y_\sigma)^{12k} \ \frac{dx_\sigma dy_\sigma}{y_\sigma^2} \\
&\leq \frac{C_1^2}{d_{\Gamma_\sigma}} \int_{\mc{F}_{\Gamma_\sigma}} |\Delta^{k - k_0} (z)|^2 (4\pi y)^{12k} \ \frac{dx dy}{y^2} \\
&\leq \frac{C_1^{2(k - k_0) + 2} }{d_{\Gamma_\sigma}} \int_{\mc{F}_{\Gamma_\sigma}} e^{-4\pi y (k-k_0)} (4\pi y)^{12k} \ \frac{dx dy}{y^2} \\
&\leq \frac{C_1^{2(k - k_0) + 2} w}{d_{\Gamma_\sigma}} \int_0^\infty e^{-4\pi y(k-k_0)} (4\pi y)^{12k} \ \frac{dx dy}{y^2} \\
&\leq \left( \frac{k}{k-k_0} \right)^{12k} e^{C_2 k},
\end{align*}
for some constant $C_2$, and where we pick a connected fundamental domain $\mc{F}_{\Gamma_\sigma}$ for $\Gamma_\sigma$ contained in a vertical strip of the form $\{ (x, y) : \beta \leq x \leq \beta + w, 0 < y \}$ for some $\beta \in \R$, where $w$ is the width of the cusp $\infty$ for $\Gamma_\sigma$. Similarly, for $t = 0, \dots, k - k_0 - 2$, 
\begin{align*}
\| \alpha_u^\sigma b_{q,p,i}^\sigma j^t \Delta^k \|_{\text{Pet}}^2 &= \frac{1}{d_{\Gamma_\sigma}} \int_{\mc{F}_{\Gamma_\sigma}} |(\alpha_u^\sigma b_{q,p,i}^\sigma j^t \Delta^k) (z)|^2 (4\pi y)^{12k} \ \frac{dx dy}{y^2} \\
&= \frac{1}{d_{\Gamma_\sigma}} \int_{\mc{F}_{\Gamma_\sigma}} |(\alpha_u^\sigma b_{q,p,i}^\sigma \Delta^{k_0 + 1})(z)|^2 |(j^t \Delta^{k - k_0 - 1}) (z)|^2 (4\pi y)^{12k} \ \frac{dx dy}{y^2} \\
&\leq \frac{C_1^2}{d_{\Gamma_\sigma}} \int_{\mc{F}_{\Gamma_\sigma}} |(j^t \Delta^{k - k_0 - 1}) (z)|^2 (4\pi y)^{12k} \ \frac{dx dy}{y^2} \\
&\leq \frac{C_1^{2t + 2k - 2k_0}}{d_{\Gamma_\sigma}} \int_{\mc{F}_{\Gamma_\sigma}} e^{-4\pi y (k - k_0 - t - 1)} (4\pi y)^{12k} \ \frac{dx dy}{y^2} \\
&\leq \frac{C_1^{2t + 2k - 2k_0} w}{d_{\Gamma_\sigma}} \int_0^\infty e^{-4\pi y (k - k_0 - t - 1)} (4\pi y)^{12k} \ \frac{dx dy}{y^2} \\
&\leq \left( \frac{k}{k - k_0 - t - 1} \right)^{12k} e^{C_3 k}
\end{align*}
for some constant $C_3$. Then
\begin{align*}
\|\alpha_u c_v\|_{k, \infty}^2 &= \frac{1}{[E : \Q]} \sum_\sigma \|\alpha_u^\sigma c_v^\sigma \Delta^k \|_{\text{Pet}}^2 \leq \left( \frac{k}{k-k_0} \right)^{12k} e^{C_2 k}, \\
\|\alpha_u b_{q,p,i}j^t\|_{k, \infty}^2 &= \frac{1}{[E : \Q]} \sum_\sigma \|\alpha_u^\sigma b_{q,p,i}^\sigma j^t \Delta^k \|_{\text{Pet}}^2 \leq \left( \frac{k}{k - k_0 - t - 1} \right)^{12k} e^{C_3 k}.
\end{align*}
Since $\alpha_u c_v$ and $\alpha_u b_{q,p,i} j^t$ are all integral and hence their $p$-norms are all at most $1$ for every finite place $p$, we have
\begin{align*}
\lambda_k( \alpha_u c_v ) &\geq -\log \| \alpha_u c_v \|_{k, \infty} \geq 6k \log \left( \frac{k - k_0}{k} \right) - C_4 k, \\
\lambda_k(\alpha_u b_{q,p,i} j^t) &\geq -\log \|\alpha_u b_{q,p,i} j^t\|_{k, \infty} \geq 6k \log \left( \frac{k - k_0 - t - 1}{k} \right) - C_4 k,
\end{align*}
for a sufficiently large constant $C_4$. 

Now, let $d_0 = \dim_\Q S_{k_0} = [E : \Q] d_0'$. For $1 \leq i \leq d_0$, any $i$ elements of the set 
\[
\bigcup_{u=1}^{[E : \Q]} \{\alpha_u c_1, \dots, \alpha_u c_{d_0'} \} 
\]
are $\Q$-linearly independent. Hence, for $1 \leq i \leq d_0$, 
\[
\lambda_{k,i} \geq 6k \log \left( \frac{k - k_0}{k} \right) - C_4 k.
\]

For $d_0 + t_0 d_{\Gamma} [E : \Q] + 1 \leq i \leq d_0 + (t_0 + 1) d_{\Gamma} [E : \Q]$, where $t_0 = 0, \dots, k - k_0 - 2$, we can take the subset
\[
\bigcup_{u=1}^{[E : \Q]} \left( \{\alpha_u c_1, \dots, \alpha_u c_{d_0'} \} \cup \bigcup_{t=0}^{t_0 - 1} \{ \alpha_u j^t b_{q,p,i} \}_{q,p,i} \right) \cup S,
\]
where $S \subseteq \bigcup_{u = 1}^{[E : \Q]} \{ \alpha_u b_{q,p,i} j^{t_0} \}$ is any subset of cardinality $i - d_0 - t_0 d_\Gamma [E : \Q]$. (If $t_0 = 0$, then the set $\bigcup_{t=0}^{t_0 - 1} \{ \alpha_u j^t b_{q,p,i} \}_{q,p,i}$ in the above union should be interpreted as the empty set.) This set of $i$ $\Q$-linearly independent elements shows that for $i$ in the above range, 
\[
\lambda_{k,i} \geq 6k \log \left( \frac{k - k_0 - t_0 - 1}{k} \right) - C_4 k. 
\]
Noting that $t_0 + 1 = \left\lceil \frac{i - d_0}{d_\Gamma [E : \Q]} \right\rceil$, we conclude that: 
\begin{itemize}
\item If $1 \leq i \leq d_0$, then 
\[
\lambda_{k,i} \geq 6k \log \left( \frac{k - k_0}{k} \right) - C_4 k.
\]

\item If $d_0 + 1 \leq i \leq d_0 + (k - k_0 - 1)d_\Gamma [E : \Q]$, then 
\[
\lambda_{k,i} \geq 6k \log \left( \frac{k - k_0 - \left\lceil \frac{i - d_0}{d_\Gamma [E : \Q]} \right\rceil}{k} \right) - C_4 k.
\]

\item If $i > d_0 + (k - k_0 - 1)d_\Gamma [E : \Q]$, the above expression still holds if we interpret the right hand side to be $-\infty$. 
\end{itemize}
This gives us the lower bounds in the proposition. 
\end{proof}

We now show that the sequence of measures $(\nu_k)_k$ is \textit{uniformly tight}: namely, given any $\varepsilon > 0$, there exists a compact set $K \subseteq \R$ for which $\nu_k(\R \setminus K) \leq \varepsilon$ for all $k$. 

\begin{Prop}
\label{uniform-tightness-prop}
The sequence of measures $(\nu_k)_k$ is uniformly tight, and hence $\nu$ is a probability measure and the vague convergence $\nu_k \to \nu$ is in fact weak convergence. Furthermore, $\nu$ has support bounded above. 
\end{Prop}
\begin{proof}
By proposition \ref{height-upper-bound-prop}, there is a constant $C$ (which we may take to be positive) such that $\lambda_{k,i}/k \leq C$ for all $k$ and all $i$, meaning that the supports of the measures $\nu_k$ are all contained in $(-\infty, C]$. Hence for any positive real number $a$, $\nu_k(\R \setminus [-a,C]) = \nu_k((-\infty, -a)) + \nu_k((C, \infty)) = \nu_k((-\infty, -a))$. Thus, to show uniform tightness, it suffices to show that for any $\varepsilon > 0$, there is a positive real number $a_1$ and a positive integer $k_1$ such that for all reals $a \geq a_1$ and all integers $k \geq k_1$, $\nu_{k}((-\infty, -a)) < \varepsilon$. (For each $i < k_1$, there is some compact set $K_i$ such that $\nu_i(\R \setminus K_i) < \varepsilon$. Letting $K' = K_1 \cup \cdots \cup K_{k_1 - 1} \cup [-a,C]$, we note that $\nu_k( \R \setminus K' ) \leq \varepsilon$ for all $k$.)

We keep the conventions from proposition \ref{lower-bounds-successive-max}. Take any $k > k_0$. Given a positive real number $a$, 
\[
\nu_k( (-\infty, -a) ) = \frac{ \#\{ i : \lambda_{k,i}/k < -a \} }{\dim_\Q S_k}. 
\]
We first restrict to counting only those $i$ with $d_0 + 1 \leq i \leq d_0 + (k - k_0 - 1)d_\Gamma [E : \Q]$, noting that the remaining $i$'s contribute at most $d_0 + d_\Gamma [E : \Q]$ to the count (which is a constant independent of $k$). Then for $i$ in the above range, proposition \ref{lower-bounds-successive-max} gives
\[
6 \log \left( \frac{k - k_0 - \left\lceil \frac{i - d_0}{d_\Gamma [E : \Q]} \right\rceil}{k} \right) - C_4 \leq \lambda_{k,i}/k < -a,
\]
which implies that
\[
i > d_0 + (k - k_0 - 1 - C_5 k e^{-a/6}) d_\Gamma [E : \Q]
\]
for $C_5 = e^{C_4/6}$. Hence,
\[
\nu_k( (-\infty, -a) ) = \frac{\#\{ i : \lambda_{k,i}/k < -a \}}{\dim_\Q S_k} \leq \frac{(C_5 k e^{-a/6} + 2)d_\Gamma [E : \Q] + d_0 + d_\Gamma [E : \Q]}{\dim_\Q S_k}.
\]
Noting that $\dim_\Q S_k$ grows linearly with $k$, we may pick $k_1$ and $a_1$ large enough so that for any $k \geq k_1$ and $a \geq a_1$, $\nu_k((-\infty, -a)) < \varepsilon$. This concludes the proof of uniform tightness.  

To see that $\nu$ is a probability measure, we note that by Prohorov's theorem, uniform tightness implies that $(\nu_k)_k$ admits a weakly convergent subsequence $(\nu_{k_m})_{m}$. If $\nu_{k_m} \to \omega$ weakly, then $\omega$ is also the vague limit of the $\nu_{k_m}$. Then by uniqueness of vague limits, we conclude that $\omega = \nu$. Since $\omega$ is a probability measure, we conclude that $\nu$ is one as well. It is a standard result that if the limit measure is also a probability measure, then vague convergence is equivalent to weak convergence, and hence $\nu_k$ converges weakly to $\nu$. Finally, it is clear that the support of $\nu$ is contained in $(-\infty, C]$. 
\end{proof}

\section{Comparison of measures - Proof of theorem \ref{measure-comparison-thm-intro}}
\label{measure-comp-proof-section}
\subsection{Notations} 
\label{conventions-comp-measures}
Given finite index subgroups $\Gamma' \subseteq \Gamma$ of $\Gamma(1)$, we say that \textit{$\Gamma' \subseteq \Gamma$ is defined over $E$}, where $E$ is a number field, if there exist smooth projective geometrically connected $E$-curves $X(\Gamma')$ and $X(\Gamma)$ such that:
\begin{itemize}
\item the base change of $X(\Gamma')$ and $X(\Gamma)$ to $\C$ give the modular curves $X(\Gamma')_\C$ and $X(\Gamma)_\C$, respectively, and 
\item there exist $E$-morphisms $\pi : X(\Gamma') \to X(\Gamma)$ and $\pi_\Gamma : X(\Gamma) \to \P_E^1$ with base change to $\C$ equal to the natural maps $X(\Gamma')_\C \to X(\Gamma)_\C$ and $X(\Gamma)_\C \to \P_\C^1$ induced by the inclusion $\Gamma' \subseteq \Gamma$, and by the $j$-function, respectively. 
\end{itemize} 
In particular, there is also an $E$-morphism $\pi_\Gamma \circ \pi : X(\Gamma') \to \P_E^1$ with base change to $\C$ equal to the natural map $X(\Gamma')_\C \to \P_\C^1$ given by the $j$-function. 

Given a finite index subgroup $\Gamma \subseteq \Gamma(1)$, we also say that \textit{$\Gamma$ is defined over $E$} to mean that $\Gamma \subseteq \Gamma(1)$ is defined over $E$. In this case, the above conditions are equivalent to the existence of models $X(\Gamma)$ and $X(\Gamma) \to \P_E^1$ over $E$ of $X(\Gamma)_\C$ and the $j$-function map $X(\Gamma)_\C \to \P_\C^1$, respectively. Note that if $\Gamma' \subseteq \Gamma$ is defined over $E$, then so are $\Gamma'$ and $\Gamma$.

For the rest of this section, suppose that $\Gamma' \subseteq \Gamma$ is defined over a number field $E$, and let $X(\Gamma'), X(\Gamma)$, and $\pi : X(\Gamma') \to X(\Gamma)$ be as above. As in \S \ref{integral-models-ss}, let $X(\Gamma)_\Z$ denote the normalization of $\P_\Z^1$ in $X(\Gamma)$ under the natural map $X(\Gamma) \to \P_E^1 \to \P_\Z^1$, and define $X(\Gamma')_\Z$ analogously. The natural map $\pi : X(\Gamma') \to X(\Gamma)$ induces a morphism $X(\Gamma')_\Z \to X(\Gamma)_\Z$ (over $\P_\Z^1$). For any choice of a desingularization $\mathscr{X}(\Gamma) \to X(\Gamma)_\Z$, there exists a desingularization $\mathscr{X}(\Gamma') \to X(\Gamma')_\Z$ along with a morphism $\pi_{\Z} : \mathscr{X}(\Gamma') \to \mathscr{X}(\Gamma)$ that extends $X(\Gamma')_\Z \to X(\Gamma)_\Z$. 

Let $\mathscr{L}' := \pi_\Z^*\mathscr{L}$, where $\mathscr{L} = \pi_{\Gamma, \Z}^* \mc{O}_{\P_\Z^1}(\ov{\infty})$ (see \S \ref{integral-models-ss}). Keeping the same conventions as in \S \ref{ad-Q-vb-rat-cusp-forms-ss}, we let 
\begin{align*}
\mathscr{M}_k &:= H^0(\mathscr{X}(\Gamma), \mathscr{L}^{\otimes k}), \quad \qquad \ \mathscr{M}_k' := H^0(\mathscr{X}(\Gamma'), \mathscr{L}'^{\otimes k}) \\
M_k &:= \mathscr{M}_k \otimes_\Z \Q, \qquad \quad \quad \qquad \ \ M_k' := \mathscr{M}_k' \otimes_\Z \Q \\
S_k &:= H^0(X(\Gamma), \mathscr{L}_\Q^{\otimes k}(-D)),  \quad \quad S_k' := H^0(X(\Gamma'), \mathscr{L'}_\Q^{\otimes k}(-D')) \\
\mathscr{S}_k &:= \mathscr{M}_k \cap S_k, \qquad \qquad \qquad \quad \mathscr{S}_k' := \mathscr{M}_k' \cap S_k',
\end{align*}
where $D$ and $D'$ are the formal sums of the pre-images of $\infty \in \P_E^1$ in $X(\Gamma)$ and $X(\Gamma')$ under the natural maps $X(\Gamma) \to \P_E^1$ and $X(\Gamma') \to \P_E^1$, respectively. We remark that these definitions are independent of the choices of the regular models $\mathscr{X}(\Gamma)$ and $\mathscr{X}(\Gamma')$. 

Given an extension $F/E$ of number fields, let $X(\Gamma)_F := X(\Gamma) \otimes_E F$ denote the base change, and let $X(\Gamma)_{F, \Z}$ denote the normalization of $\P_\Z^1$ under the natural map $X(\Gamma)_F \to \P_\Z^1$. The map $X(\Gamma)_F \to X(\Gamma)$ extends to a map between the normalizations $X(\Gamma)_{F,\Z} \to X(\Gamma)_\Z$. Next, let $\mathscr{X}(\Gamma)_F$ be a desingularization of $X(\Gamma)_{F, \Z}$ over $F$ that admits a morphism to $\mathscr{X}(\Gamma)$ extending the natural map $X(\Gamma)_{\Z, F} \to X(\Gamma)_\Z$. Let $\mathscr{M}_{F,k}$ denote the global sections of the pullback of $\mathscr{L}^{\otimes k}$ to $\mathscr{X}(\Gamma)_F$, and define $M_{F,k}, S_{F,k}$, and $\mathscr{S}_{F,k}$ in the obvious manner. 

For any place $v$ of $\Q$, we let $\| \cdot \|_{k,v}, \| \cdot \|_{k, v}'$, and $\| \cdot \|_{F,k,v}$ denote the local norms on $S_k \otimes_\Q \C_v, S_k' \otimes_\Q \C_v$, and $S_{F,k} \otimes_\Q \C_v$, respectively. Finally, let $\lambda_k, \lambda_k'$, and $\lambda_{F,k}$ be the naive adelic height functions on $S_k, S_k'$, and $S_{F,k}$, respectively. %Finally, let $\nu_k, \nu_k'$, and $\nu_{F,k}$ denote the probability measures associated to the naive adelic successive maxima on $S_k, S_k'$, and $S_{F,k}$, respectively, and let $\nu, \nu'$, and $\nu_F$ denote their weak limits. 

We adopt the convention that we drop the subscript $E$ for objects defined over $E$ (as in \S \ref{integral-models-ss}), but include the subscript for objects over subfields or field extensions of $E$.

\subsection{Main results}
\begin{Lemma}
\label{global-sections-intersection}
Let $E$ be a number field, $X$ and $Y$ smooth projective integral curves over $E$, and $\pi : Y \to X$ a non-constant map of curves over $E$. Suppose $\mathscr{X}$ and $\mathscr{Y}$ are regular projective models of $X$ and $Y$, respectively, over $\spec(\mc{O}_E)$, and let $\pi_\Z : \mathscr{Y} \to \mathscr{X}$ be a $\spec(\mc{O}_E)$-morphism extending $\pi$. For any line bundle $\mathscr{L}$ on $\mathscr{X}$ with $L := \mathscr{L}|_X$, we have the equality
\[
H^0(\mathscr{X}, \mathscr{L}) = H^0(X, L) \cap H^0(\mathscr{Y}, \pi_\Z^* \mathscr{L}),
\]
where the intersection takes place in $H^0(Y, \pi^* L)$. 
\end{Lemma}
\begin{proof}
First let $\mathscr{X}' \xrightarrow{\eta} \mathscr{X}$ denote the normalization of $\mathscr{X}$ in $Y$ with respect to the map $Y \to X \to \mathscr{X}$. Then $\mathscr{X}'$ is a normal arithmetic surface, which admits a birational map $\mathscr{Y} \dashrightarrow \mathscr{X}'$ (inducing the identity on the generic fiber $Y$). By \cite{Liu} chapter 9.2, theorem 2.7, there is a projective birational morphism $\mathscr{Y}' \xrightarrow{\alpha} \mathscr{Y}$, with $\mathscr{Y}'$ regular, and a birational morphism $\mathscr{Y}' \xrightarrow{\beta} \mathscr{X}'$ lifting $\mathscr{Y} \dashrightarrow \mathscr{X}'$. The natural pullback morphisms induce the equality $H^0(\mathscr{X}', \eta^*\mathscr{L}) = H^0(\mathscr{Y}', \beta^* \eta^* \mathscr{L}) = H^0(\mathscr{Y}, \pi_\Z^* \mathscr{L})$ (since $\eta \circ \beta = \pi_\Z \circ \alpha$) in $H^0(Y, \pi^* L)$. Hence, it suffices to show that $H^0(\mathscr{X}, \mathscr{L}) = H^0(X, L) \cap H^0(\mathscr{X}', \eta^*\mathscr{L})$. 

We only need to show $\supseteq$. Take $s \in H^0(X, L) \cap H^0(\mathscr{X}', \eta^*\mathscr{L})$. Let $\spec(A) \subseteq \mathscr{X}$ be an affine open subset where $\mathscr{L}$ is trivial. Then $L$ is trivial on $A \otimes_\Z \Q$, and $\eta^* \mathscr{L}$ is trivial on $\eta^{-1}(\spec(A)) = \spec(A')$, where $A'$ is the integral closure of $A$ in the rational function field $\kappa(Y)$. The section $s$ then corresponds to an element in $A' \cap (A \otimes_\Z \Q) = A \subseteq \kappa(X)$, since $A$ is integrally closed in its fraction field $\kappa(X)$. 
\end{proof}

\begin{Lemma}
\label{norm-compatibility-lemma}
Let $\Gamma' \subseteq \Gamma$ be finite index subgroups of $\Gamma(1)$ defined over a number field $E$, and let $F/E$ be a finite extension. Then with the notation in \S \ref{conventions-comp-measures}, the local norms respect the inclusions
\begin{enumerate}[(a)]
\item $S_k \otimes_\Q \C_v \to S_k' \otimes_\Q \C_v$, and 
\item $S_k \otimes_\Q \C_v \to S_{F,k} \otimes_\Q \C_v$
\end{enumerate}
for all places $v$ of $\Q$. 
\end{Lemma}
\begin{proof}
First let $v = p$ be a finite place. The inclusion of finite free $\Z$-modules $\mathscr{M}_k \subseteq \mathscr{M}'_k$ implies that there exists a $\Z$-basis $\{b_1, \dots, b_{d'} \}$ of $\mathscr{M}'_k$ for which there exist integers $n_1, \dots, n_d$ such that $n_1 b_1, \dots, n_d b_d$ is a $\Z$-basis for $\mathscr{M}_k$. Since $\mathscr{M}_k = M_k \cap \mathscr{M}'_k$ by lemma \ref{global-sections-intersection}, we must have $b_i \in \mathscr{M}_k$, and hence $n_i = 1$ for all $i = 1, \dots, d$. Given $s \in M_k \otimes_\Q \C_p$, if $s = \sum_{i=1}^d \alpha_i b_i$ with $\alpha_i \in \C_p$, then $\| s \|_{k,p} = \| s \|_{k,p}' = \max \{ |\alpha_i| \}_{i=1}^d$, by the definition of the local norms at $p$. A similar argument also shows that $\| s \|_{k,p} = \| s \|_{F,k,p}$. 

Now suppose $v = \infty$. For any $\sigma : E \to \C$, the base change of $\pi : X(\Gamma') \to X(\Gamma)$ by $\sigma$ gives $\pi_\sigma : X(\Gamma'_\sigma)_\C \to X(\Gamma_\sigma)_\C$ for finite index subgroups $\Gamma'_\sigma \subseteq \Gamma_\sigma \subseteq \Gamma(1)$. Now given any $f = (f_\sigma)_\sigma \in S_k \otimes_\Q \C \subseteq S_k' \otimes_\Q \C$, the Petersson norm of the cusp form $f_\sigma \Delta^k \in S_{12k}(\Gamma_\sigma)_\C \subseteq S_{12k}(\Gamma'_\sigma)_\C$ is independent of the groups $\Gamma_\sigma$ and $\Gamma'_\sigma$, we get $\| f \|_{k, \infty} = \| f \|_{k, \infty}'$.  

Finally, we note that for $f$ as above, 
\begin{align*}
\| f \|_{F, k, \infty}^2 &= \frac{1}{[F : \Q]} \sum_{\tau : F \to \C} \| f^\tau \Delta^k \|_{\text{Pet}}^2 \\
&= \frac{1}{[F : \Q]} \sum_{\sigma : E \to \C} \sum_{\substack{\tau \\ \tau|_E = \sigma}} \| f^\tau \Delta^k \|_{\text{Pet}}^2 \\
&= \frac{1}{[E : \Q]} \sum_{\sigma : E \to \C} \| f^\sigma \Delta^k \|_{\text{Pet}}^2 = \| f \|_{k, \infty}^2. 
\end{align*}
\end{proof}

Now suppose $\Gamma' \trianglelefteq \Gamma$ is normal, and that $\pi : X(\Gamma') \to X(\Gamma)$ is Galois over $E$. Since $X(\Gamma')_\Z$ is also the normalization of $X(\Gamma)_\Z$ in $X(\Gamma')$, the universal property of normalization implies that every element $\tau \in \text{Aut}(X(\Gamma')/X(\Gamma))$ lifts uniquely to an element $\wt{\tau} \in \text{Aut}(X(\Gamma')_\Z/X(\Gamma)_\Z)$. Let the elements of $\Aut(X(\Gamma')/X(\Gamma))$ be denoted $\tau_j$, and let $(\cdot)|_{\tau_j}$ denote the pullback map on sections induced by $\tau_j$.   

Suppose also that $F/E$ is a Galois extension of number fields. Since $X(\Gamma)$ is geometrically integral over $E$, $\Gal(X(\Gamma)_F/X(\Gamma)) = \Gal(F/E)^{\text{op}}$ (where the ``op'' means that we take the opposite group). For $\alpha_j \in \Gal(F/E)$, let $\alpha_j^*$ denote the corresponding element of $\Gal(X(\Gamma)_F/X(\Gamma))$. Again by the universal property of normalization, $\alpha_j^*$ extends uniquely to an automorphism $\Aut(X(\Gamma)_{\Z,F}/X(\Gamma)_\Z)$. We denote by $(\cdot)^{\alpha_j}$ the pullback map on sections induced by $\alpha_j^*$. 

\begin{Lemma}
\label{maxima-comparison-lemma}
Let $\Gamma' \trianglelefteq \Gamma$ be defined over $E$ and suppose that $\pi : X(\Gamma') \to X(\Gamma)$ is Galois over $E$. Let $F/E$ be a Galois extension of number fields. Then with the conventions of \S \ref{conventions-comp-measures}, we have
\begin{enumerate}[(a)]
\item $S_k^a \subseteq S_k'^a \cap S_k \subseteq S_k^{a-\log [\Gamma' : \Gamma]}$, and 
\item $S_k^a \subseteq S_{F,k}^a \cap S_k \subseteq S_k^{a - \log [F : E]}$
\end{enumerate}
for all $a \in \R$. 
\end{Lemma}
\begin{proof}
The first containment for both (a) and (b) is the content of lemma \ref{norm-compatibility-lemma}. For the second containment, take any $f \in S_k'^a \cap S_k$ (resp. $s \in S_{F,k}^a \cap S_k$). Then $f = \sum_i g_i$ for $g_i \in S_k'$ with $\lambda_k'(g_i) \geq a$ (resp. $s = \sum_i t_i$ for $t_i \in S_{F,k}$ with $\lambda_{F,k}(t_i) \geq a$). Let $h_i := \sum_j g_i|_{\tau_j}$ (resp. $u_i := \sum_j t_i^{\alpha_j}$). Then we claim that $h_i \in S_k$ with $\lambda_k(h_i) \geq a - \log [\Gamma' : \Gamma]$ (resp. $u_i \in S_k$ with $\lambda_k(u_i) \geq a - \log [F : E]$) for all $i$, from which the conclusion follows. 

First suppose $v = p$ is a finite place. Given any $\Z$-basis $\{b_k\}$ of $\mathscr{M}'_k$, if $g_i = \sum_k \alpha_k b_k$ with $\alpha_k \in \Q$, then $\| g_i \|_{k,p}' = \max\{|\alpha_k|\}_k$. Now since $\tau_j$ extends uniquely to $\Aut(X(\Gamma')_\Z/X(\Gamma)_\Z)$, the pullback map $(\cdot)|_{\tau_j} : M'_k \to M'_k$ restricts to an automorphism of the integral sections $\mathscr{M}'_k$ (see the remarks in \S \ref{integral-models-ss}). Hence $\{b_k|_{\tau_j}\}$ is also a $\Z$-basis of $\mathscr{M}'_k$, from which we conclude that $\|g_i|_{\tau_j}\|'_{k,p} = \|g_i\|'_{k,p}$ for all $\tau_j$. Hence,
\[
\|h_i\|_{k,p} = \|h_i\|'_{k,p} = \left\| \sum_j g_i|_{\tau_j} \right\|'_{k,p} \leq \max_j \|g_i|_{\tau_j} \|'_{k,p} = \|g_i\|'_{k,p}. 
\]
A similar argument shows that $\| u_i \|_{k,p} \leq \| t_i \|_{F,k,p}$.

Next, suppose $v = \infty$. First, we address (a). We have the base change diagram by $\sigma : E \to \C$
\[
\begin{tikzcd}
X(\Gamma'_\sigma)_\C \arrow[r, "\gamma_j"] \arrow[d] & X(\Gamma'_\sigma)_\C \arrow[d] \\
X(\Gamma') \arrow[r, "\tau_j"] & X(\Gamma')
\end{tikzcd}
\]
where the top horizontal map $\gamma_j \in \Aut( X(\Gamma'_\sigma)_\C/X(\Gamma_\sigma)_\C)$ corresponds to the automorphism of the modular curve $X(\Gamma'_\sigma)_\C$ induced by the coset $\gamma_j\Gamma' \in \Gamma/\Gamma'$ for some $\gamma_j \in \Gamma$. Hence for all $i$ and $j$, we have $(g_i|_{\tau_j})^\sigma = g_i^\sigma|_{\gamma_j}$ (where $|_{\gamma_j}$ denotes the usual slash operator for modular forms, which in this case is simply pre-composition by $\gamma_j$). Then we have
\begin{align*}
\|h_i\|_{k, \infty}^2 = \|h_i\|_{k, \infty}'^2 &= \frac{1}{[E : \Q]} \sum_{\sigma : E \to \C} \| h_i^\sigma \Delta^k \|_{\text{Pet}}^2 \\
&\leq \frac{1}{[E : \Q]} \frac{1}{d_{\Gamma'}} \frac{d_{\Gamma'}}{d_\Gamma} \sum_\sigma \int_{\mc{F}_{\Gamma'}} \sum_j \left| ((g_i|_{\tau_j})^{\sigma} \Delta^k) (z) \right|^2  (4\pi y)^{12k} \ \frac{dx dy}{y^2} \\
&= \frac{1}{[E : \Q]} \frac{1}{d_{\Gamma'}} \frac{d_{\Gamma'}}{d_\Gamma} \sum_\sigma \int_{\mc{F}_{\Gamma'}} \sum_j \left| ((g_i^\sigma|_{\gamma_j}) \Delta^k) (z) \right|^2  (4\pi y)^{12k} \ \frac{dx dy}{y^2} \\
&= \frac{1}{[E : \Q]} \frac{d_{\Gamma'}}{d_\Gamma} \sum_j \sum_\sigma \| g_i^\sigma \Delta^k \|_{\text{Pet}}^2 \\
&= \left( \frac{d_{\Gamma'}}{d_\Gamma} \right)^2 \|g_i\|_{k, \infty}'^2. 
\end{align*}
where we use Cauchy-Schwartz inequality in the second line, and the invariance of the Petersson inner product with respect to the slash operator for modular forms in the fourth line. We conclude that 
\[
\lambda_k(h_i) = - \sum_v \log \|h_i\|_{k,v} \geq - \sum_v \log \|g_i\|_{k,v}' - \log \left(\frac{d_{\Gamma'}}{d_\Gamma} \right) = \lambda'_k(g_i) - \log [\Gamma' : \Gamma]. 
\]

Finally, we address (b). For each $\sigma : E \to \C$, fix a lift $\wt{\sigma} : F \to \C$. Then all lifts of $\sigma$ to $F$ are given by $\wt{\sigma} \circ \alpha_j$ as $\alpha_j$ varies over $\Gal(F/E)$. 
\begin{align*}
\|u_i\|_{k, \infty}^2 = \frac{1}{[E : \Q]} \sum_{\sigma : E \to \C} \| u_i^\sigma \Delta^k \|_{\text{Pet}}^2 &= \frac{1}{[E : \Q]} \sum_{\sigma : E \to \C} \| u_i^{\wt{\sigma}} \Delta^k \|_{\text{Pet}}^2 \\
&= \frac{1}{[E : \Q]} \sum_{\sigma : E \to \C} \left \| \sum_j (t_i^{\alpha_j})^{\wt{\sigma}} \Delta^k \right \|_{\text{Pet}}^2 \\
%&= \frac{1}{[E : \Q]} \sum_{\sigma : E \to \C} \left \| \sum_j t_i^{\wt{\sigma} \circ \alpha_j} \Delta^k \right \|_{\text{Pet}}^2 \\
&\leq \frac{[F : E]}{[E : \Q]} \sum_{\sigma : E \to \C} \sum_j \left \| t_i^{\wt{\sigma} \circ \alpha_j} \Delta^k \right \|_{\text{Pet}}^2 \\
&= [F : E]^2 \frac{1}{[F : \Q]} \sum_{\tau : F \to \C} \left \| t_i^{\tau} \Delta^k \right \|_{\text{Pet}}^2 \\
&= [F : E]^2 \|t_i\|_{F,k,\infty}^2.
\end{align*}
Hence, we conclude that
\[
\lambda_k(u_i) \geq \lambda_{F,k}(t_i) - \log[F : E],
\]
as required. 
\end{proof}
%%%%% detailed calc 
\iffalse
\begin{align*}
\|h_i\|_{k, \infty}^2 = \|h_i\|_{k, \infty}'^2 &= \sum_{\sigma : F \to \C} \| h_i^\sigma \Delta^k \|_{\text{Pet}}^2 \\
&= \frac{1}{d_{\Gamma'}} \sum_\sigma \int_{\mc{F}_{\Gamma'}} |(h_i^\sigma \Delta^k)(z_\sigma)|^2 (4\pi y_\sigma)^{12k} \ \frac{dx_\sigma dy_\sigma}{y_\sigma^2} \\
&= \frac{1}{d_{\Gamma'}} \sum_\sigma \int_{\mc{F}_{\Gamma'}} \left| \sum_j ((g_i|_{\tau_j})^{\sigma} \Delta^k) (z_\sigma) \right|^2 (4\pi y_\sigma)^{12k} \ \frac{dx_\sigma dy_\sigma}{y_\sigma^2} \\
&\leq \frac{1}{d_{\Gamma'}} \frac{d_{\Gamma'}}{d_\Gamma} \sum_\sigma \int_{\mc{F}_{\Gamma'}} \sum_j \left| ((g_i|_{\tau_j})^{\sigma} \Delta^k) (z_\sigma) \right|^2  (4\pi y_\sigma)^{12k} \ \frac{dx_\sigma dy_\sigma}{y_\sigma^2} \\
&= \frac{1}{d_\Gamma} \sum_\sigma \int_{\mc{F}_{\Gamma'}} \sum_j \left| ((g_i^\sigma|_{\gamma_j}) \Delta^k) (z_\sigma) \right|^2  (4\pi y_\sigma)^{12k} \ \frac{dx_\sigma dy_\sigma}{y_\sigma^2} \\
&= \frac{1}{d_\Gamma} \sum_j \sum_\sigma \int_{\mc{F}_{\Gamma'}} \left| (g_i^\sigma \Delta^k) (z_\sigma) \right|^2  (4\pi y_\sigma)^{12k} \ \frac{dx_\sigma dy_\sigma}{y_\sigma^2} \\
&= \frac{d_{\Gamma'}}{d_\Gamma} \sum_j \sum_\sigma \| g_i^\sigma \Delta^k \|_{\text{Pet}}^2 \\
&= \left( \frac{d_{\Gamma'}}{d_\Gamma} \right)^2 \|g_i\|_{k, \infty}'^2. 
\end{align*}
\fi
%%%%%%%%%

\subsection{Proof of theorem \ref{measure-comparison-thm-intro}}
\label{proof-measure-comp}
\begin{proof}[Proof of theorem \ref{measure-comparison-thm-intro}]
Let $E$ and $E_0$ be the fields of constants of $X(\Gamma')$ and $X(\Gamma)$, respectively. Taking global sections of the structure sheaves for the morphism $\pi_{\Gamma', \Gamma} : X(\Gamma') \to X(\Gamma)$ yields an inclusion $E_0 \hookrightarrow E$, and hence an $E$-morphism $X(\Gamma') \to X(\Gamma) \otimes_{E_0} E$ such that its base change to $\C$ is equal to the natural map $X(\Gamma')_\C \to X(\Gamma)_\C$ by assumption. Hence, $\Gamma' \subseteq \Gamma$ is defined over $E$. 

Let $\Gamma''$ be a finite index subgroup of $\Gamma'$ with $\Gamma'' \trianglelefteq \Gamma$. Let $F$ be a finite extension of $E$ that is Galois over $E_0$, and suppose that $\Gamma''$ is defined over $F$, with model $X(\Gamma'')_F$. By extending $F$ if necessary, we may also assume that there exist $F$-morphisms $X(\Gamma'')_F \to X(\Gamma')_F \to X(\Gamma)_F$ with base change to $\C$ equal to the natural maps $X(\Gamma'')_\C \to X(\Gamma')_\C \to X(\Gamma)_\C$, and that the maps $X(\Gamma'')_F \to X(\Gamma')_F$ and $X(\Gamma'')_F \to X(\Gamma)_F$ are Galois covers of curves. Then repeated application of lemma \ref{maxima-comparison-lemma} gives
\begin{align*}
S_{E_0, k}^a \subseteq S_k^a \cap S_{E_0, k} \subseteq S_k'^a \cap S_{E_0,k} &\subseteq S_{F,k}'^a \cap S_{E_0,k} \\
&\subseteq S_{F,k}''^a \cap S_{E_0, k} \subseteq S_{F,k}^{a - \log[\Gamma'' : \Gamma]} \cap S_{E_0, k} \subseteq S_{E_0,k}^{a - \log [\Gamma'' : \Gamma] - \log [F : E_0]}. 
\end{align*}
In particular, we have 
\begin{equation}
\label{maxima-comparison-subspace}
S_{E_0, k}^a \subseteq S_k'^{a} \cap S_{E_0, k} \subseteq S_{E_0,k}^{a - \log [\Gamma'' : \Gamma] - \log [F : E_0]}. 
\end{equation}
Let $\lambda_{E_0, k, i}'$ denote the successive maxima of $S_{E_0, k}$ with respect to the subspace filtration $S_k'^a \cap S_{E_0, k}$, and let 
\[
\nu_{E_0, k}' := \frac{1}{\dim_\Q S_{E_0, k}} \sum_{i=1}^{\dim_\Q S_{E_0, k}} \delta_{\frac{1}{k} \lambda_{E_0, k, i}'}. 
\]
Equation \ref{maxima-comparison-subspace} gives
\[
\lambda_{E_0, k, i} \leq \lambda_{E_0, k, i}' \leq \lambda_{E_0, k, i} + \log [\Gamma'' : \Gamma] + \log [F : E_0]
\]
for all $i = 1, \dots, \dim_\Q S_{E_0, k}$. Consequently, $\nu_{E_0, k}' \to \nu_{E_0}$ weakly. 

Consider now the short exact sequence 
\[
0 \to S_{E_0, k} \to S_k' \to W_k := S_k'/S_{E_0, k} \to 0.
\]
If we equip $W_k$ with the quotient filtration as on page 16 of \cite{Chen-1}, then by proposition 1.2.5 of \cite{Chen-1}, there is a Borel probability measure $\omega_k$ on $\R$ such that 
\[
(\dim_\Q S_k') \cdot \nu_k' = (\dim_\Q S_{E_0, k}) \cdot \nu_{E_0, k}' + (\dim_\Q W_k) \cdot \omega_k,
\]
and hence
\[
\omega_k = \frac{\dim_\Q S_k'}{\dim_\Q W_k} \cdot \nu_k' - \frac{\dim_\Q S_{E_0, k}}{\dim_\Q W_k} \cdot \nu_{E_0, k}'. 
\]
Taking the limit as $k \to \infty$, we get that $\omega_k$ converges weakly to the Borel probability measure 
\[
\omega := \frac{[E : E_0][\Gamma : \Gamma']} {[E : E_0][\Gamma : \Gamma'] - 1} \cdot \nu' - \frac{1} {[E : E_0] [\Gamma : \Gamma'] - 1} \cdot \nu_{E_0}.
\]
Rearranging, we get 
\[
\nu' = \frac{1}{[E : E_0][\Gamma : \Gamma']} \cdot \nu_{E_0} + \left( 1 - \frac{1}{[E : E_0][\Gamma : \Gamma']} \right) \cdot \omega.
\]
Finally, since $[E : E_0][\Gamma : \Gamma'] = \text{deg}(\pi_{\Gamma', \Gamma})$, we get the desired conclusion. 
\end{proof}

\begin{Cor}
\label{unboundedness-general-corollary}
Assume the setup in theorem \ref{main-theorem-statement-intro}. Then the support of the limit measure $\nu$ is bounded above and unbounded below. 
\end{Cor}
\begin{proof}
That the support of $\nu$ is bounded above follows from lemma \ref{height-upper-bound-prop}. The modular curve $X(1)_\C$ for $\Gamma(1)$ has a model $X(1)_\Q$ over $\Q$ that we identify with $\P_\Q^1$ via the $j$-function. We have the morphism $\pi_\Gamma : X(\Gamma) \to \P_\Q^1$ associated to the inclusion $\Gamma \subseteq \Gamma(1)$, and hence we may apply theorem \ref{measure-comparison-thm-intro} to get
\[
\nu = \frac{1}{\text{deg}(\pi_\Gamma)} \cdot \nu_{\Gamma(1), \Q} + \left( 1 - \frac{1}{\text{deg}(\pi_\Gamma)} \right) \cdot \omega,
\]
where $\nu_{\Gamma(1), \Q}$ is the limit measure associated to the successive maxima of the spaces of $\Q$-rational cusp forms of level $\Gamma(1)$ and weight $12k$ as in theorem 3.2.2 of \cite{TGS}, and $\omega$ is a Borel probability measure on $\R$. By part (ii) of the same theorem, the support of $\nu_{\Gamma(1), \Q}$ is not bounded below. Consequently, the support of $\nu$ is unbounded below as well. 
\end{proof}

\bibliographystyle{amsalpha}
\bibliography{references}

\end{document}